%% file: arxiv.tex
\documentclass[a4paper]{article}
\usepackage{fullpage}
\usepackage{microtype} 

\usepackage{amssymb}

\usepackage{amsmath}
\usepackage{yhmath} 
\usepackage{units}  
\newcommand{\Reals}{\mathbb{R}}
\newcommand{\seg}{\overline}
\newcommand{\arc}{\wideparen}
\newcommand{\diam}{\mathrm{diam}}
\newcommand{\umbra}{U}
\newcommand{\deep}{\hat{U}}
\newcommand{\dists}{d_{s}}
\newcommand{\ds}{\delta^{\ast}}
\newcommand{\dht}{\hat{\delta}}
\newcommand{\dsk}{\delta^{\ast}_{k}}
\newcommand{\dskii}{\delta^{\ast}_{2}}
\newcommand{\dsiii}{\delta^{\ast}_{3}}
\newcommand{\dsvi}{\delta^{\ast}_{6}}
\newcommand{\as}{a^{\ast}}
\newcommand{\aht}{\hat{a}}
\newcommand{\ask}{a^{\ast}_{k}}
\newcommand{\asii}{a^{\ast}_{2}}
\newcommand{\asvi}{a^{\ast}_{6}}
\newcommand{\eps}{\varepsilon}
\newcommand{\strip}{\mathfrak{S}}
\newcommand{\reg}{\mathfrak{R}}
\newcommand{\midline}{\mathfrak{M}}
\newcommand{\boundary}{\mathfrak{B}}
\newcommand{\aii}{\nicefrac{\alpha}{2}}
\newcommand{\xii}{\nicefrac{\xi}{2}}
\newcommand{\dsii}{\nicefrac{\ds}{2}}
\newcommand{\pii}{\nicefrac{\pi}{2}}
\newcommand{\muk}{\mu_k}
\newcommand{\muii}{\mu_2}
\newcommand{\muiii}{\mu_3}
\newcommand{\muvi}{\mu_6}
\newcommand{\longk}{\lambda_k}
\newcommand{\shortk}{\sigma_k}
\newcommand{\longvi}{\lambda_6}
\newcommand{\shortvi}{\sigma_6}

\usepackage{xspace}
\usepackage[shortlabels]{enumitem}
\usepackage{graphicx}
\usepackage{color}
\usepackage{url}

\usepackage{amsthm}

\usepackage{hyperref}

\hypersetup{%
  breaklinks,%
  ocgcolorlinks,
  colorlinks=true,%
  linkcolor=[rgb]{0.45,0.0,0.0},%
  citecolor=[rgb]{0,0,0.45}
}

\newtheorem{theorem}{Theorem}
\newtheorem{lemma}[theorem]{Lemma}
\newtheorem{corollary}[theorem]{Corollary}
\theoremstyle{plain}
\newtheorem{observation}[theorem]{Observation}

\let\geq\geqslant
\let\leq\leqslant

\title{Shortcuts for the Circle%
  \thanks{SWB is supported by Basic Science Research
    Program through the National Research Foundation of Korea (NRF)
    funded by the Ministry of Education (2015R1D1A1A01057220).  MdB is
    supported by the Netherlands’ Organisation for Scientific Research
    (NWO) under project no.~024.002.003.  OC~is supported by NRF grant
    2011-0030044 (SRC-GAIA) funded by the government of Korea. JG is
    supported under Australian Research Council's Discovery Projects
    funding scheme (project number DP150101134).}}

\author{Sang Won Bae%
  \thanks{Kyonggi University, Korea}
  \and
  Mark de Berg%
  \thanks{TU Eindhoven, The Netherlands}
  \and
  Otfried Cheong%
  \thanks{KAIST, Korea}
  \and
  Joachim Gudmundsson%
  \thanks{University of Sydney, Australia}
  \and
  Christos Levcopoulos%
  \thanks{Lund University, Sweden}}
\makeatletter \ifx\showkeyslabelformat\@undefined\else
\renewcommand{\showkeyslabelformat}[1]{\normalfont\tiny\ttfamily#1}
\fi
\partopsep\z@ \textfloatsep 10pt plus 1pt minus 4pt
\def\section{\@startsection {section}{1}{\z@}{-3.5ex plus -1ex minus
    -.2ex}{2.3ex plus .2ex}{\large\bf}}
\def\subsection{\@startsection{subsection}{2}{\z@}{-3.25ex plus -1ex
    minus -.2ex}{1.5ex plus .2ex}{\normalsize\bf}}
\def\@fnsymbol#1{\ensuremath{\ifcase#1\or *\or 1\or 2\or 3\or 4\or
    5\or 6\or 7 \or 8\ or 9 \or 10\or 11 \else\@ctrerr\fi}}
\makeatother

\begin{document}

\maketitle

\begin{abstract}
  Let $C$ be the unit circle in $\Reals^2$. We can view $C$ as a plane
  graph whose vertices are all the points on $C$, and the distance
  between any two points on $C$ is the length of the smaller arc
  between them. We consider a graph augmentation problem on $C$, where
  we want to place $k\geq 1$ \emph{shortcuts} on $C$ such that the
  diameter of the resulting graph is minimized.

  We analyze for each $k$ with $1\leq k\leq 7$ what the optimal set of
  shortcuts is. Interestingly, the minimum diameter one can obtain is
  not a strictly decreasing function of~$k$. For example, with seven
  shortcuts one cannot obtain a smaller diameter than with six
  shortcuts. Finally, we prove that the optimal diameter is $2 +
  \Theta(1/k^{\nicefrac{2}{3}})$ for any~$k$.
\end{abstract}

\section{Introduction}
\label{sec:intro}

Graph augmentation problems have received considerable attention over
the years. The goal in such problems is typically to add extra edges
to a given graph~$G$ in order to improve some quality measure. One
natural quality measure is the (vertex- or edge-)connectivity
of~$G$. This has led to work where one tries to find the minimum
number of edges that can be added to the graph to ensure it is
$k$-connected, for a desired value of~$k$. Another natural measure is
the diameter of $G$, that is, the maximum distance between any pair of
vertices. The goal then becomes to reduce the diameter as much as
possible by adding a given number of edges, or to achieve a given
diameter with a small number of extra edges; see for example the
papers by Erd\"{o}s, R\'{e}nyi, and
S\'{o}s~\cite{er-optg-62,ers-optg-66}.

Chung and Garey~\cite{cg-dbag-84} studied this problem for the special
case where the original graph is the $n$-vertex cycle.  They showed
that if $k$ edges are added, then the diameter of the resulting graph
is at least $\frac{n}{k+2}-3$ for even~$k$ and $\frac{n}{k+1}-3$ for
odd~$k$, and that there is a way to add~$k$ edges so that the
resulting graph has diameter at most $\frac{n}{k+2}-1$ for even~$k$
and $\frac{n}{k+1}-1$ for odd~$k$.
(For paths, slightly better bounds are known~\cite{sbl-died-87}.)

The algorithmic problem of finding a set of $k\geq 1$ edges that
minimizes the diameter of the augmented graph was first asked by
Chung~\cite{c-dgopn-87} in 1987. Since then many papers have
considered the problem for general graphs,
see~\cite{bgp-ianrg-12,fggm-agmd-15,ks-bdmcg-07,lms-mcbdb-92,sbl-died-87}.
Gro{\ss}e et al.~\cite{grosse} were the first to consider the diameter
minimization problem in the geometric setting where the graph is
embedded in the Euclidean plane. They presented an $O(n \log^3 n)$
time algorithm that determined the optimal shortcut that minimizes the
diameter of a polygonal path with $n$ vertices. The running time was
later improved to $O(n \log n)$ by Wang~\cite{w-iadoap-16}.

In the above papers only the discrete setting is considered, that is,
shortcuts connect two vertices and the diameter is measured between
vertices. In the continuous setting all points along the edges of the
network are taken into account when placing a shortcut and when
measuring distances in the augmented network.  In the continuous
setting, Yang~\cite{yang} studied the special case of adding a single
shortcut to a polygonal path and gave several approximation algorithms
for the problem. De Carufel et al.~\cite{decarufel} considered the
problem for paths and cycles. For paths they showed that an optimal
shortcut can be determined in linear time. For cycles they showed that
a single shortcut can never decrease the diameter, while two shortcuts
always suffice. They also proved that for convex cycles the optimal
pair of shortcuts can be computed in linear time. Recently,
C\'{a}ceres et al.~\cite{caceres} gave a polynomial time algorithm
that can determine whether a plane geometric network admits a
reduction of the continuous diameter by adding a single shortcut.

We are interested in a geometric continuous variant of this problem.
Let~$C$ be a unit circle in the plane. We define the
\emph{distance}~$d(p,q)$ between two points $p,q\in C$ to be the
length of the smaller arc along~$C$ that connects $p$ to $q$. Thus the
diameter of $C$ in this metric is $\pi$. We now want to add a number
of \emph{shortcuts}---a shortcut is a chord of $C$---to improve the
diameter. Here the distance $d_S(p,q)$ between $p$ and $q$ for a given
collection~$S$ of shortcuts is defined as the length of the shortest
path between $p$ and $q$ that can travel along $C$ and along the
shortcuts where, if two shortcuts intersect in their interior, we do
not allow the path to switch from one shortcut to the other at the
intersection point. In other words, if the path uses a shortcut, it
has to traverse it completely. Note that if we view the circle $C$ as
a graph with infinitely many vertices (namely all points on $C$) where
the graph distance is the distance along~$C$, then adding shortcuts
corresponds to adding edges to the graph. For a set $S$ of shortcuts,
define $\diam(S) := \max_{p,q\in C} d_S(p,q)$ to be the diameter of
the resulting “graph.” We are interested in the following question:
given $k$, the number of shortcuts we are allowed to add, what is the
best diameter we can achieve? In other words, we are interested in the
quantity $\diam(k) := \inf_{|S|=k} \diam(S)$.

It is obvious that 
$\pi = \diam(0) \geq \diam(1) \geq \cdots \geq \diam(k) \geq \cdots
\geq \lim_{k\rightarrow \infty} \diam(k) = 2$. \smallbreak

Our main results are as follows.
\begin{itemize}
\item For $1\leq k\leq 7$, we determine $\diam(k)$ exactly. Our
  results show that $\diam(k)$ is not strictly decreasing as a
  function of $k$. This not only holds at the very beginning---it is
  easy to see that $\diam(1)=\diam(0)$---but, interestingly also for
  certain larger values of $k$. In particular, we show that
  $\diam(7)=\diam(6)$.
\item We have $\diam(8) < \diam(7)$.
\item We show that $\diam(k) = 2 + \Theta(1/k^{\frac{2}{3}})$.
\end{itemize}

We rely on a number of numerical calculations.  A Python script that
performs these calculations can be found at
\url{http://github.com/otfried/circle-shortcuts}, and
the output of the
script is included as Appendix~\ref{sec:calculations}.


\section{The umbra and the region of a shortcut}

A shortcut~$s$ is a chord of~$C$.  A shortcut of length $a = |s| \in
[0, 2]$ spans an angle of~$\alpha(a) \in [0,\pi]$, where $\alpha(a) :=
2\arcsin\big(\frac{a}{2}\big)$.  The following function $\delta: [0,2]
\mapsto [0,\pii-1]$ will play a key role in our arguments:
\begin{align*}
\delta(a) & :=\frac{\alpha(a) - a}{2} = \arcsin\left(\frac{a}{2}\right) -
\frac{a}{2}.
\end{align*}
Note that both $\alpha(a)$ and $\delta(a)$ are increasing and convex
functions, and $\alpha(a) = a + 2\delta(a)$. See
\figurename~\ref{fig:notations}. To simplify the notation, we will
allow shortcuts themselves as the function argument, with the
understanding that $\alpha(s) = \alpha(|s|)$ and $\delta(s) =
\delta(|s|)$.

We parameterize the points on the circle~$C$ using their polar angle
in~$[0, 2\pi)$.  For a shortcut~$s$ with endpoints~$u$ and~$v$ we will
  write $s = uv$ if the \emph{counter-clockwise} arc~$\arc{uv}$ is the
  \emph{shorter} arc of~$C$ connecting~$u$ and~$v$.  Only for $|s| =
  2$, we have $s = uv = vu$; in this case~$u$ and~$v$ are
  \emph{antipodal points}, that is $v = u + \pi$.

The \emph{inner umbra} of a shortcut~$s = uv$ is the
arc~$\arc{u_1v_1}$ where $u_1 = u + \delta(s)$ and $v_1 = v -
\delta(s)$. The \emph{outer umbra} is the set of antipodal points of
the inner umbra, that is the arc~$\arc{u'_1v'_1}$ where $x' = x +
\pi$.  Together they form the \emph{umbra}~$U(s)$ of~$s$.  Since
$\alpha(s) = |s| + 2\delta(s)$, the inner and outer umbra have
length~$|s|$.  The \emph{radiance} of~$s$ consists of the two
arcs~$\arc{vu'}$ and~$\arc{v'u}$.  For $|s| = 2$, we cannot
distinguish inner and outer umbra, and the radiance consists of two
isolated points, see
\figurename~\ref{fig:notations}.
\begin{figure}[thb]
  \centerline{\includegraphics{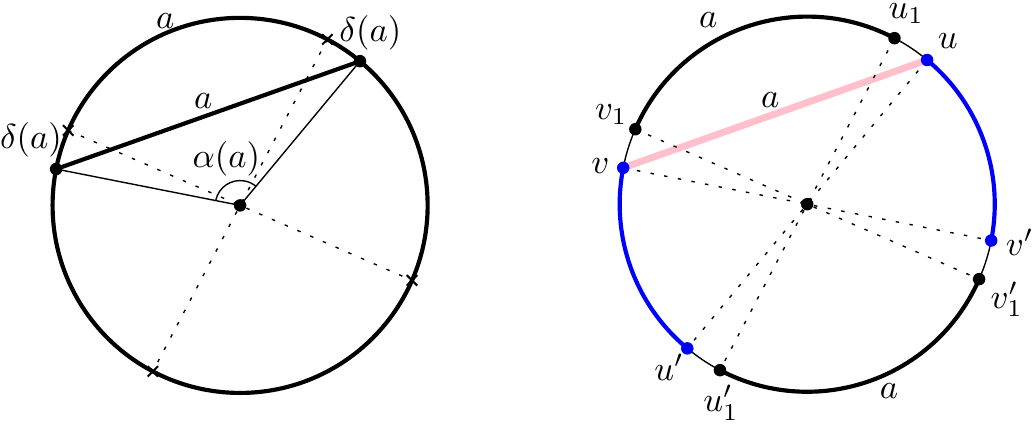}}
  \caption{For a shortcut $s$ of length $a = |s|$, (left)
    $\alpha(a)=a+2\delta(a)$ and (right) the umbra $\umbra(s)$
    (consisting of two arcs of length $a$ in thick black) and radiance
    (in thick blue).}
  \label{fig:notations}
\end{figure}

Let $p \in U(s)$.  Then a path going from~$p$ to one endpoint of~$s$
and traversing the shortcut is at least as long as going directly
from~$p$ to the other endpoint---so the shortcut is not useful.  This
gives us the following observation:
\begin{observation}
  \label{obs:umbra}
  Given a set $S$ of shortcuts, if the shortest path~$\gamma$ from~$p$
  to~$q$ uses shortcuts~$s_1, s_2, \dots, s_m \in S$ in this order,
  then $p \not\in U(s_1)$ and $q \not\in U(s_m)$.
  If $\gamma$ traverses~$s_{i}$ from its endpoint~$u_{i}$ to its other
  endpoint~$v_{i}$, then $v_{i} \notin U(s_{i+1})$ and
  $u_{i+1} \notin U(s_i)$ for $i = 1, \dots, m-1$.
\end{observation}
(For the boundary cases, we will assume that a shortest path uses the
minimum number of shortcuts possible.)  An immediate implication is
that one shortcut alone cannot help to improve the diameter, that is,
$\diam(1) = \diam(0) = \pi$.

Another useful observation is the following (remember that $d(p,q) =
\min(|\arc{pq}|, |\arc{qp}|)$ is the distance along~$C$ without
shortcuts):
\begin{observation}
  \label{obs:deltas}
  Given a set $S$ of shortcuts,
  if the shortest path from~$p$ to~$q$ uses the set of
  shortcuts~$\{s_1, s_2, \dots, s_m\}\subseteq S$,
  then $d_{S}(p,q) \geq d(p,q) - 2\sum_{i=1}^{m}\delta(s_{i})$.
\end{observation}
Indeed, if $\gamma$ is the shortest path, we can replace each shortcut
$s_{i}$ by walking along the circle instead, increasing the path
length by exactly~$2\delta(s_{i})$.

Observation~\ref{obs:umbra} does not exclude the possibility that the
shortest path for $p, q \in U(s)$ uses $s$ as an intermediate shortcut
(not the first or last one). We therefore define the \emph{deep
  umbra}~$\deep(s)$ of~$s$ as the set of points~$p$ in the inner umbra
of~$s$ such that~$|\seg{pu}| + |s| \geq d(p,v)$, where $u$ is the
endpoint of~$s$ closer to~$p$ and $v$ is the other endpoint.
\begin{observation}
  \label{obs:deep}
  If $p\in \deep(s)$ or $q\in\deep(s)$ then the shortest path from~$p$
  to~$q$ does not use shortcut~$s$.
\end{observation}
The following lemma shows that the deep umbra covers nearly the entire
inner umbra. 
\begin{lemma}
  \label{lem:deep}
  The deep umbra~$\deep(s)$ of a shortcut~$s$ is non-empty and has
  length at least $|s| - 4\delta(\delta(s)) > |s| - 0.02$.
\end{lemma}
\begin{proof}
  Let $s = uv$ and let $u_2 = u + \alpha(\delta(s))$ and
  $v_2 = v - \alpha(\delta(s))$.  Since $\delta(s) < \alpha(\delta(s))
  < \alpha(s)/2$, the arc $\arc{u_2v_2}$ is non-empty and lies inside
  the inner umbra of~$s$.

  We claim that $\arc{u_2v_2} \subset \deep(s)$. Indeed, for $p \in
  \arc{u_2v_2}$ such that $d(p, u) < d(p, v)$, we have
  $|\seg{pu}| \geq |\seg{u_2u}| = \delta(s)$, and so
  $d(p, v) = |\arc{pv}| \leq |\arc{u_2v}| = \alpha(s) -
  \alpha(\delta(s)) < |s| + \delta(s) \leq |s| + |\seg{pu}|$.

  The arc length of $\arc{u_2v_2}\subset \deep(s)$ is~$\alpha(s) -
  2\alpha(\delta(s))$.  Since $\alpha(x) = x + 2\delta(x)$,
  this is $|s| + 2\delta(s) - 2(\delta(s) + 2\delta(\delta(s))) =
  |s| - 4\delta(\delta(s))$.

  Finally, we observe that the function $x \mapsto \delta(\delta(x))$
  is increasing and $\delta(\delta(2)) < 0.005$.
\end{proof}

Let us now fix a target diameter of the form~$\pi - \ds$, for some
$\ds \in [0, \pi - 2]$.  To achieve the target diameter, pairs of
points $p, q \in C$ that span an angle of at most $\pi - \ds$ do not
need a shortcut, so it suffices to consider pairs of points $p, q \in
C$ where $q = p + \pi + \xi$, for~$-\ds \leq \xi \leq \ds$.  We
represent these point pairs by the rectangle~$\strip(\ds) = [0,2\pi]
\times [-\ds,+\ds]$, where $(\theta,\xi)$ corresponds to the pair of
points $p = \theta-\xii$ and $q = \theta + \pi + \xii$, as illustrated
in \figurename~\ref{fig:notations2}.  So the counter-clockwise angle
from $p$ to $q$ is~$\pi + \xi$.

\begin{figure}[thb]
  \centerline{\includegraphics[width=.8\textwidth]{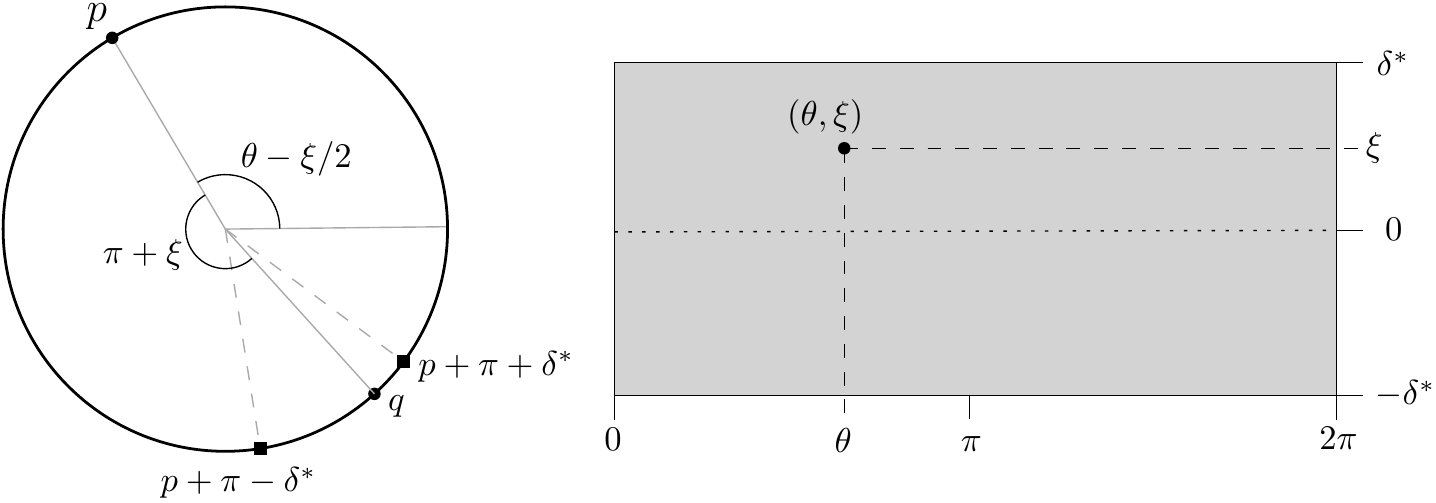}}
  \caption{(left) $(\theta,\xi)$ corresponds to the pair of points $p
    = \theta-\xii$ and $q = \theta + \pi + \xii$. (right)
    $\strip(\ds)$ represents all pair of points $p = \theta-\xii$ and
    $q = \theta + \pi + \xii$ with $-\ds \leq \xi \leq \ds$. }
  \label{fig:notations2}
\end{figure}

$\strip(\ds)$ is topologically a cylinder: the right edge~$\theta =
2\pi$ is identified with the left edge~$\theta = 0$.  Furthermore, if
the point pair~$(p,q) \in C\times C$ corresponds to $(\theta,\xi)$,
then the point pair~$(q,p)$ corresponds to $(\theta+\pi, -\xi)$.
Since $d_{S}(p,q) = d_{S}(q,p)$, we could therefore identify the
middle segment $\theta = \pi$ with the left edge~$\theta = 0$, but
with opposite orientation, resulting in a M\"{o}bius strip topology.
As will become clear shortly, for our purposes it is easier to work
with the cylinder topology, but keep in mind that, for instance, the
upper boundary $\xi = \ds$ and the lower boundary $\xi= -\ds$
of~$\strip(\ds)$ really represent the same point pairs.

For a shortcut~$s$, we define the region~$\reg(s,\ds) \subset
\strip(\ds)$ consisting of those pairs $(\theta, \xi) \in \strip$
where $\dists(\theta-\xii, \theta+\pi+\xii) \leq \pi - \ds$.  (In the
following, we will use~$\dists(p, q)$ for $d_{\{s\}}(p,q)$.)

Let us fix a shortcut~$s$ of length~$a > 0$, and let $\alpha =
\alpha(a)$ and $\delta = \delta(a)$.  Rotating a shortcut around the
origin means translating~$\reg(s, \ds)$ horizontally in (the
cylinder)~$\strip(\ds)$.  We can thus choose $s$ to be vertical and
connect the points~$-\aii$ and~$\aii$.  This implies that the umbra
of~$s$ consists of the two intervals~$[-\aii+\delta, \aii-\delta]$ and
$[\pi-\aii+\delta, \pi+\aii-\delta]$.  The radiance of $s$ consists of
the two intervals~$[\aii,\pi-\aii]$ and~$[\pi+\aii,2\pi-\aii]$.

The following function gives the length of the path from~$p$ to~$q$
that uses the shortcut~$s$ from top to bottom, that is, from the
point~$\aii$ to~$-\aii$:
\[
f(\theta,\xi) := |\aii - p| + a + |q - (2\pi-\aii)|, \qquad
\text{where} \quad (p,q) = (\theta-\xii, \theta+\pi+\xii).
\]
By the observation about the M\"{o}bius strip topology above, it
suffices to understand~$\reg(s,\ds)$ for~$0 \leq \theta \leq \pi$.  We
claim that for $0 \leq \theta \leq \pi$ we have $ \dists(p, q) < \pi -
\ds$ if and only if $f(\theta, \xi) < \pi - \ds$.

This is clearly true if the shortest path from~$p$ to~$q$ uses~$s$
from top to bottom, or not at all, because the length of the shorter
circle arc between~$p$ and~$q$ is $\pi - |\xi| \geq \pi - \ds$.  It
remains to consider the case when the shortest path uses~$s$ from
bottom to top.  This can only happen when $p$ is closer to the bottom
end of~$s$ than to its top end---in other words, when $\pi < p <
2\pi$.  Since $0 \leq \theta \leq \pi$ and $p = \theta - \xii$, this
implies either $\theta < \dsii$ and $\xi > 2\theta$, or $\theta > \pi
- \dsii$ and $\xi < -2(\pi - \theta)$.  Since $q = \theta + \pi
+ \xii$, the first case implies $\pi \leq q \leq \pi + \ds < 2\pi$,
while the second case implies $\pi < 2\pi - \ds \leq q \leq 2\pi$.  In
both cases, $q$ lies closer to the bottom end of the shortcut than to
its top end, a contradiction to the shortcut being used from bottom to
top to go from~$p$ to~$q$.

\medskip

It follows that for $0 \leq \theta \leq \pi$, we have $(\theta, \xi)
\in \reg(s, \ds)$ if and only if $f(\theta, \xi) \leq \pi - \ds$.  To
analyze~$f$, we partition the rectangle~$[0,\pi]\times[-\ds,\ds]$ into
regions, depending on the signs of $\aii - p$ and $q - (2\pi-\aii)$.
First, we have $p < \aii$ if and only if $\xi > 2\theta-\alpha$.  This
is the lightly shaded region~$A$ above the blue line in
Figure~\ref{fig:abcd}(left).
\begin{figure}[thb]
  \centerline{\includegraphics[scale=0.9]{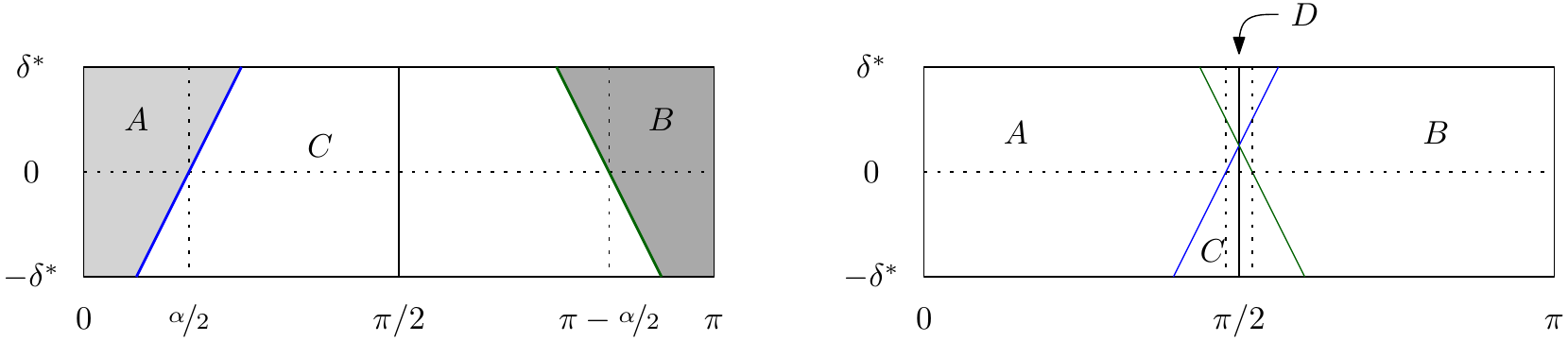}}
  \caption{The regions~$A, B, C, D$.}
  \label{fig:abcd}
\end{figure}
Second, we have $q > 2\pi-\aii$ if and only if $\xi > 2\pi - \alpha -
2\theta$. This is the darker region~$B$ above the green line in
Figure~\ref{fig:abcd}(left). If the two regions do not intersect then
we get three regions as shown in Figure~\ref{fig:abcd}(left).
Otherwise, if $\alpha > \pi-\ds$, or equivalently, $\ds > \pi - a -
2\delta$, then the regions intersect and we get four regions as
illustrated in Figure~\ref{fig:abcd}(right).  We now study
$\reg(s,\ds)$ independently for each of the three or four regions.

In region~$A$, we have $p < \aii$ and $q < 2\pi-\aii$.
It follows that
\begin{align*}
f(\theta,\xi) & = \aii - p + \alpha - 2\delta + 2\pi - \aii - q \\
& = -\theta+\xii + \alpha - 2\delta + 2\pi - \theta - \pi -\xii \\
& = \pi - 2\delta + 2(\aii - \theta).
\end{align*}
This implies that $f(\theta,\xi) \leq \pi - \ds$ if and only if
$\theta \geq \aii + \dsii - \delta = \nicefrac{a}{2} + \dsii$.  This
is the blue area as shown in \figurename~\ref{fig:abcd2}.
\begin{figure}[thb]
  \centerline{\includegraphics[width=0.9\textwidth,page=2]{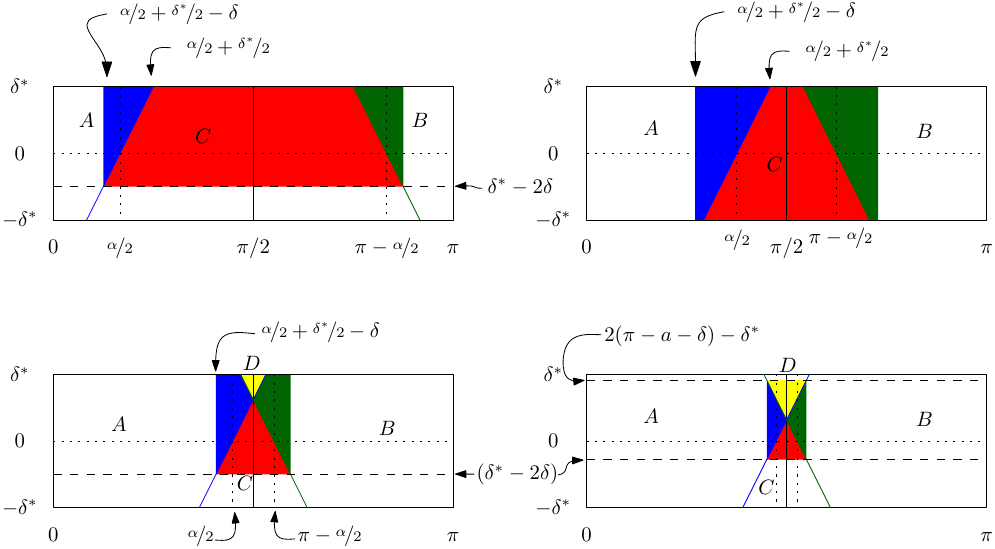}}
  \caption{The region~$\reg(s,\ds)$ in four different situations.}
  \label{fig:abcd2}
\end{figure}

In region~$B$, we have $p \geq \aii$ and $q > 2\pi-\aii$. This implies
\begin{align*}
f(\theta,\xi) & = p - \aii + \alpha - 2\delta + q - 2\pi + \aii \\
& = \theta-\xii + \alpha - 2\delta + \theta + \pi +\xii -2\pi \\
& = 2(\theta - (\pi-\aii)) + \pi - 2\delta,
\end{align*}
and so we have $f(\theta,\xi) \leq \pi-\ds$ if and only if $\theta
\leq \pi-\aii-\dsii+\delta$.  This is the green area in
\figurename~\ref{fig:abcd2}.

Next, in region~$C$, we have $p \geq \aii$ and $q \leq 2\pi-\aii$.
Therefore,
\begin{align*}
f(\theta,\xi) & = p - \aii + \alpha - 2\delta + 2\pi - \aii - q\\
& = \theta-\xii - 2\delta + 2\pi - \theta - \pi -\xii \\
& = \pi - 2\delta - \xi.
\end{align*}
We have $f(\theta,\xi) \leq \pi-\ds$ if and only if $\xi \geq \ds -
2\delta$.  This is the red area in \figurename~\ref{fig:abcd2}.

When $\alpha > \pi-\ds$ regions $A$ and $B$ intersect in region~$D$,
as shown in Figure~\ref{fig:abcd}(right). In region~$D$ we have $p < \aii$
and $q > 2\pi-\aii$, and therefore
\begin{align*}
f(\theta,\xi) & = \aii - p + \alpha - 2\delta + q - 2\pi + \aii\\
& = 2\alpha - 2\delta - 2\pi -\theta+\xii + \theta + \pi + \xii \\
& = 2(\alpha - \delta) - \pi + \xi \\
& = 2(a + \delta) - \pi + \xi,
\end{align*}
since $\alpha - \delta = a + \delta$.  Thus, we have $f(\theta, \xi)
\leq \pi-\ds$ if and only if $\xi \leq 2(\pi - a - \delta) - \ds$.  This is the
yellow area in region~$D$ in \figurename~\ref{fig:abcd2}. There are
two cases that can occur, as is shown on the bottom left and bottom
right of \figurename~\ref{fig:abcd2}. We postpone the discussion of
these cases to the proof of the following lemma, which summarizes our
discussions above.
\begin{lemma}
  \label{lem:region}
  Let $\ds \in [0,\pi - 2]$, and let $s$ be a shortcut of length $a\in
  (0,2]$.  Then, the region~$\reg(s,\ds)$ of $s$ in the
  cylinder~$\strip(\ds) = [0,2\pi] \times [-\ds,+\ds]$ forms two identical
  rectangles whose width is exactly $\pi - a - \ds$ and whose height is
  \[ \begin{cases}
    2\ds            & \text{if $\ds \leq \delta(a)$} \\
    2\delta(a)      & \text{if  $\ds > \delta(a)$
      and $\ds \leq \pi - a - \delta(a)$} \\
    2(\pi - a -\ds) & \text{otherwise.}
  \end{cases}
  \]
\end{lemma}
\begin{proof}
  We consider each case separately.  Let $\alpha = \alpha(a)$ and
  $\delta = \delta(a)$.

  First assume that $\ds \leq \delta$.  Since $a \leq 2$ and $\delta
  \leq \pii - 1$, we have
  \[
  \pi - a - \delta \geq \pii - 1 = \delta(2) \geq \delta \geq \ds.
  \]
  Thus, we have $\ds \leq \pi - a - \delta$ and $2(\pi - a - \delta) -
  \ds \geq \ds$.  This implies that the region~$\reg(s,\ds)$ contains
  the whole region~$D$ if $D$ is nonempty.  In this case, $\reg(s,\ds)$
  forms two identical rectangles that span the entire height $2\ds$
  of~$\strip(\ds)$, as shown in Figure~\ref{fig:abcd2}(top right).  Its
  width is determined by
  \[
  (\pi - \aii -\dsii + \delta) - (\aii + \dsii - \delta) = \pi - \alpha
  + 2\delta - \ds = \pi - a - \ds,
  \]
  since the left and right boundaries of~$\reg(s,\ds)$ are $\aii +
  \dsii - \delta \leq \theta \leq \pi - \aii -\dsii + \delta$.

  Second, suppose that $\ds > \delta$ and $\ds \leq \pi - a - \delta$.
  Then, we again have $2(\pi - a - \delta) - \ds \geq \ds$ and the
  region~$\reg(s,\ds)$ contains the whole region~$D$ if $D$ is
  nonempty.  In this case, $\reg(s,\ds)$ forms two rectangles
  in~$\strip(\ds)$. One touches the top boundary of~$\strip(\ds)$ (as
  shown in \figurename~\ref{fig:abcd2}(top left, bottom left)), the
  other one the bottom boundary.  The height of the rectangle is exactly
  $2\delta$ as $\reg(s,\ds)$ in region~$C$ is delimited by $\xi \geq
  \ds - 2\delta$, while the width of the rectangle is $\pi - a - \ds$ as
  above.

  Finally, in the remaining case $\ds > \pi - a - \delta$.  Then, $D$
  must be nonempty as $\ds > \pi - a - 2\delta = \pi - \alpha$.  In
  this case, the region~$\reg(s,\ds)$ in region~$D$ is delimited by
  $\xi \leq 2(\pi - a - \delta) - \ds$.  Since $\ds > \pi - a -
  \delta$, we have a strict inequality $2(\pi - a - \delta) - \ds <
  \ds$.  Figure~\ref{fig:abcd2}(bottom right) illustrates this case.
  Observe on one hand that the horizontal width and the vertical
  height of~$\reg(s,\ds) \cap D$ (the yellow area in the figure) is
  exactly $\pi - a - \ds$.  On the other hand, the width and the
  height of~$\reg(s,\ds) \cap C$ (the red area in the figure) is also
  equal to $\pi - a -\ds$ since $(2(\pi - a - \delta) - \ds) - (\ds -
  2\delta) = 2(\pi - a - \ds)$.  Thus, $\reg(s,\ds) \cap (C \cup D)$
  is of width $\pi - a - \ds$ and thus fits in between $\aii + \dsii -
  \delta \leq \theta \leq \pi - \aii -\dsii + \delta$.  So, the
  region~$\reg(s,\ds)$ of $s$ again forms two rectangles
  in~$\strip(\ds)$. Neither of them touches a boundary
  of~$\strip(\ds)$, both have width~$\pi - a - \ds$ and height~$2(\pi
  - a - \ds)$.
\end{proof}
Note that if $\ds \leq \delta(2) = \pii - 1$, then it always holds
that $\ds \leq \pi - a - \delta(a)$ for any $0 \leq a \leq 2$ since
$\pi - a - \delta(a) \geq \pi - 2 - \delta(2) = \delta(2) \geq \ds$.
Hence, the last case of Lemma~\ref{lem:region} where $\ds > \pi - a -
\delta(a)$ only happens when $\ds > \delta(2) = \pii - 1$.

We will also be interested in the length of the intersection
of~$\reg(s,\ds)$ with the middle line $\midline = \{ \xi = 0\}$ of
$\strip(\ds)$ and its upper boundary $\boundary = \{ \xi = \ds\}$.
Note that both $\midline$ and $\boundary$ have length $2\pi$.
We have the following corollary to Lemma~\ref{lem:region}:
\begin{corollary}
  \label{coro:region}
  Let $\ds\in [0,\pi - 2]$, and let $s$ be a shortcut of length $a
  \in (0,2]$.  Then
  \[
  |\midline \cap \reg(s,\ds)| =
  \begin{cases}
    2(\pi - a - \ds) & \text{if $\delta(a) \geq \dsii$} \\
    0                & \text{otherwise.}
  \end{cases}
  \]
  and
  \[
  |\boundary \cap \reg(s,\ds)| =
  \begin{cases}
    2(\pi - a - \ds)  & \text{if $\delta(a) \geq \ds$} \\
    \pi - a - \ds     & \text{if $\delta(a) < \ds \leq \pi - a -
      \delta(a)$} \\ 
    0               & \text{otherwise.}
  \end{cases}
  \]
\end{corollary}

\section{Up to five shortcuts}
\label{sec:up2five}

In this section we derive the exact value of~$\diam(k)$ for $k \in
\{2, 3, 4, 5\}$ and the unique optimal configuration of shortcuts in
each case.  The proof is quite easy, comparing the areas of~$\reg(s,
\ds)$ with the area of~$\strip(\ds)$, if one assumes that the shortest
path between any pair of points uses at most one shortcut.  Showing
that using a combination of shortcuts does not help takes considerable
additional effort.

\subsection{Using only one shortcut}

Again we consider a target diameter of the form~$\pi - \ds$, with $\ds
\in [0,\pi - 2]$.  By Lemma~\ref{lem:region}, the region~$\reg(s,\ds)$
of a shortcut~$s$ of length~$a$ consists of two rectangles of width
$\pi - a - \ds$ and height~$2\delta(a)$ for $\delta(a) < \ds$, and
height $2\ds$ for $\delta(a) \geq \ds$.  We define $\as$ such that
$\delta(\as) = \ds$, or $\as = 2$ when $\ds > \delta(2)$.

Then the area~$A(a, \ds)$ of~$\reg(s,\ds)$ is
\[
A(a, \ds) = \left\{ \begin{array}{lcl}
4\ds(\pi - a - \ds) & \text{for} & a > \as \\
4\delta(a)(\pi - a - \ds) & \text{for} & a \leq \as
\end{array}
\right.
\]
\begin{lemma}
  \label{lem:area-max}
  For fixed $\ds \leq 0.7$, the function~$a \mapsto A(a, \ds)$ is
  increasing for $a \leq \as$ and decreasing for $a \geq \as$. Its
  maximum value is $A(\as, \ds) = 4\ds(\pi - \as -\ds)$.
\end{lemma}
\begin{proof}
  For $a \geq \as$, the function $a \mapsto A(a, \ds)$ is a decreasing
  linear function.  To verify that $A(a, \ds)$ is increasing for $a
  \leq \as$, we consider the derivative $\frac{d}{da} (\delta(a)(\pi -
  a - \ds))$ of the function $a \mapsto \delta(a)(\pi - a - \ds)$ for
  any fixed $\ds$ with $0 < \ds \leq 0.7$.  We have
  \begin{align*}
     \frac{d}{da} (\delta(a)(\pi - a - \ds)) & =
       (\pi - a - \ds) \Big( \frac{1}{\sqrt{4-a^2}} -
     \frac{1}{2}\Big) - \delta(a)\\
     & \geq (2.4 - a) \Big( \frac{1}{\sqrt{4-a^2}} -
     \frac{1}{2}\Big) - \delta(a)\\
     & = \frac{2.4-a}{\sqrt{4-a^2}} + a - 1.2 - \arcsin\big(
     \frac{a}{2} \big) = g\big(\frac{a}{2}\big),
  \end{align*}
  where we define
  \begin{align*}
    g(x) & = \frac{1.2-x}{\sqrt{1-x^2}} + 2x - 1.2 - \arcsin x.
  \end{align*}
  We will prove the lemma by showing that $g(x) > 0$ for $0 < x < 1$.
  Consider again the derivative:
  \begin{align*}
    g'(x) & 
    = \frac{1}{(1-x^{2})^{\nicefrac 32}}\big(
    x^2 + 1.2x - 2 + 2(1-x^2)^{\nicefrac 32}\big).
  \end{align*}
  Set $h(x) = x^2 + 1.2x - 2 + 2(1-x^2)^{\nicefrac 32}$. The function
  $h(x)$ is continuous on the interval~$[0,1]$ and has only a single
  real root at $x = 0$, so $h(\nicefrac 12) > 0$ implies~$h(x) > 0$
  for $0 < x < 1$.  This implies that $g'(x) > 0$ for $0<x<1$, so
  $g(0) = 0$ implies $g(x) > 0$ for $0 < x < 1$, completing the proof.
\end{proof}

Let $k \in \{2,3,4,5\}$.  Since $a \mapsto a + \delta(a)$ is an
increasing function that maps $[0,2]$ to~$[0,\pii+1]$, there is a
unique $\ask$ that solves the equation
\[
\ask + \delta(\ask) = \frac{k-1}{k} \pi.
\]
We set $\dsk := \delta(\ask)$, and will show that this number
determines the optimal diameter for $k$~shortcuts.
\tablename~\ref{table:diameter} shows the numerical values.  For
completeness, we already include the case $k=6$ in the table by
setting $\asvi = 2$.

\begin{table}[hbt]
  \centerline{\begin{tabular}{c||cc|c|c}
      $k$ & $\ask$ & $\dsk$ & $\diam(S) = \pi - \dsk$ & $\mu_k$\\
      \hline
      2  & 1.4782 & 0.0926 & 3.0490 & 1.2219 \\
      3  & 1.8435 & 0.2509 & 2.8907 & 1.5943 \\
      4  & 1.9619 & 0.3943 & 2.7473 & 1.7623 \\
      5  & 1.9969 & 0.5164 & 2.6252 & 1.8526 \\
      6  & 2.0000 & 0.5708 & 2.5708 & 1.8828
  \end{tabular}}
  \caption{The values $\ask$, $\dsk$, $\pi - \dsk$, and $\muk$.}
  \label{table:diameter}
\end{table}

\begin{lemma}
  \label{lem:optimal-no-combination}
  For $k \in \{2, 3, 4, 5\}$ there is a set~$S$ of~$k$ shortcuts that
  achieves $\diam(S) = \pi - \dsk$.\\ Assuming that no pair of points
  uses more than one shortcut, this is optimal and the solution is
  unique up to rotation.
\end{lemma}
\begin{proof}
By Lemma~\ref{lem:region}, the region~$\reg(s, \dsk)$ of a shortcut~$s$ of
length~$|s| = \ask$ consists of two rectangles of height~$2\dsk$ and
width $\pi - (\ask + \dsk) = \pi/k$.  Each rectangle covers the entire
height of~$\strip(\ds)$, and by rotating~$s$ about the origin we can
translate the rectangles anywhere inside~$\strip(\ds)$. This implies
that we can use $k$~such rectangles to cover the range $0 \leq \theta
\leq \pi$.  Then for every $(\theta, \xi) \in \strip(\ds)$ there is a
shortcut~$s$ such that $\dists(\theta - \xii, \theta+\pi+\xii) \leq
\pi - \dsk$, and $\diam(S) = \pi - \dsk$.
\figurename~\ref{fig:opt_upto5} shows the resulting configurations.

Assume now that a set $S = \{s_1, \dots, s_k\}$ of $k$ shortcuts is
given with $\diam(S) \leq \pi - \ds$, where $\ds \geq \dsk$, and that
no pair of points uses more than one shortcut.  This implies that the
regions $\reg(s_i, \ds)$ must entirely cover the strip~$\strip(\ds)$,
and in particular
\[
\sum_{i = 1}^{k} A(|s_i|, \ds) \geq 4\ds\pi.
\]
If we choose $\as$ such that $\delta(\as) = \ds$, then $\as \geq \ask$.
By Lemma~\ref{lem:area-max} we have
\[
A(|s_i|, \ds) \leq A(\as, \ds) = 4\ds(\pi - \as - \ds).
\]
From $k A(\as, \ds) \geq 4\ds\pi$ we have $k(\pi - \as - \ds) \geq
\pi$, or $\as + \ds \leq \frac{k-1}{k} \pi$, which implies $\as =
\ask$ and $\ds = \dsk$. But then the regions $\reg(s_{i}, \dsk)$ must
be non-overlapping, and the solution is unique up to rotation.
\end{proof}

\begin{figure}[thb]
  \centerline{\includegraphics[width=.95\textwidth]{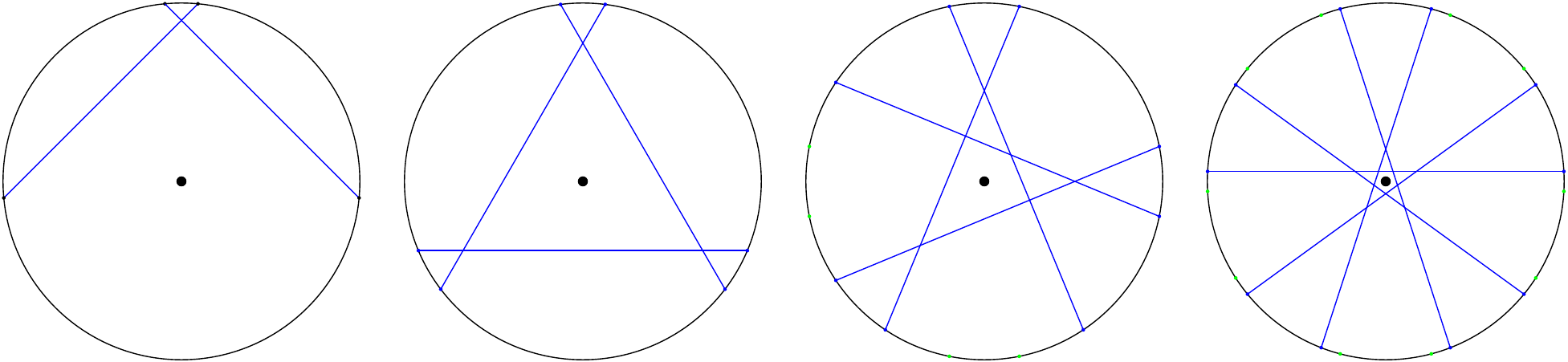}}
  \caption{The optimal shortcut configurations for $k=2, 3, 4, 5$.}
  \label{fig:opt_upto5}
\end{figure}

It remains to show that the configurations in
\figurename~\ref{fig:opt_upto5} are optimal even if combinations of
shortcuts can be used.  The case~$k=2$ is somewhat special and handled
first.

We start by defining $\muk \in [0, 2]$ to be such that $\delta(\muk) =
\dsk/2$.  \tablename~\ref{table:diameter} shows the numerical
values. By Lemma~\ref{lem:region}, $\reg(s, \dsk)$ intersects the
middle line~$\midline$ if and only if $|s| \geq \mu_{k}$.  In other
words, for two antipodal points~$p$ and $q$ we can have $\dists(p,q)
\leq \pi - \dsk$ only if $|s| \geq \muk$.

\subsection{Optimality for two shortcuts}
\label{ssec:two-optimal}

\begin{lemma}
  \label{lem:two-optimal}
  If $S$ is a set of two shortcuts that achieves diameter
  $\diam(S) \leq \pi - \dskii$, then $S$ is identical to the configuration
  of \figurename~\ref{fig:opt_upto5} up to rotation.
\end{lemma}
\begin{proof}
  Let $S = \{s_1, s_2\}$ with $|s_1| \leq |s_2|$.  Let $p$ and~$q$ be
  the midpoints of the inner and outer umbra of~$s_2$. The shortest
  path between~$p$ and~$q$ cannot use~$s_2$ at all by
  Observation~\ref{obs:umbra}, so~$d_{s_1}(p,q) \leq \pi - \dskii$.
  This implies $|s_1| \geq \muii \approx 1.2219$.  Since $\dskii
  \approx 0.0926 < \muii/2$, the interval $[q - \dskii, q + \dskii]$
  lies in~$\umbra(s_2)$, and so we have $d_{s_1}(p,q') \leq \pi -
  \dskii$ for all~$q' \in [q - \dskii, q + \dskii]$. This implies
  $|s_1| \geq \asii$.

  We next observe that $\umbra(s_1) \cap \umbra(s_2) =
  \emptyset$. Otherwise, Observation~\ref{obs:umbra} applied to an
  antipodal pair in $\umbra(s_1) \cap \umbra(s_2)$ implies $\diam(S)
  = \pi$, a contradiction.

  The two arcs between the inner and outer umbras of $s_1$ have length
  $\pi - |s_1| \leq \pi - \asii$.  The inner umbra $\umbra(s_2)$ has
  length~$|s_2| \geq \asii$ and lies in one of these arcs. That leaves
  a gap of most $\pi - 2\asii = 2\dskii$ between the two inner umbras
  (by definition of~$\asii$, we have $\asii + \dskii = \pii$).  Since
  $\delta(s_2) \geq \delta(s_1) \geq \dskii$, this implies that the
  two shortcuts intersect, see \figurename~\ref{fig:k2}.

  Let $x$ be the length of overlap of the arcs of~$s_1$ and~$s_2$,
  that is, $x = |\arc{u_1 v_2}|$ in \figurename~\ref{fig:k2}.  Any
  path that uses both $s_1$ and $s_2$ has length at least $|s_1| +
  |s_2| + x \geq 2\asii + x = \pi - 2\dskii + x$.  This is bounded by
  $\pi - \dskii$ only if $x \leq \dskii$.  But then the arc
  $\arc{v_1u_2}$ has length at most
  \[
  2\pi - \alpha(s_1) - \alpha(s_2) + x \leq
  2\pi - 2(\asii + 2\dskii) + \dskii =
  \pi + (\pi - 2\asii) - 3\dskii = \pi - \dskii,
  \]
  and there is no reason to use the two shortcuts at all.  It follows
  that there is no pair of points that uses more than one shortcut,
  and Lemma~\ref{lem:optimal-no-combination} implies the claim.
\end{proof}

\begin{figure}[hbt]
  \centerline{\includegraphics{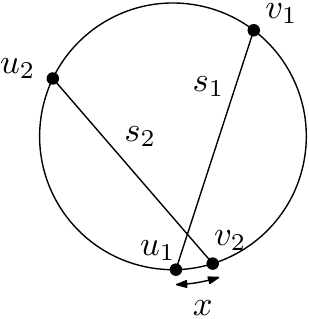}}
  \caption{The two shortcuts must intersect.}
  \label{fig:k2}
\end{figure}

\subsection{Antipodal pairs cannot use combinations of shortcuts}
\label{ssec:antipodal}

The key to the general proof for $3\leq k \leq 6$ is the following
lemma.  We prove it for two separate cases: $k=3$, and
$k\in\{4,5,6\}$.
\begin{lemma}
  \label{lem:antipodal}
  Let $S$ be a set of $k$ shortcuts for $k \in \{3, 4, 5, 6\}$ such
  that $\diam(S) \leq \pi - \dsk$.  Then there is no antipodal pair of
  points $p, q \in C$ such that the path of length $d_{S}(p, q)$ uses
  more than one shortcut.
\end{lemma}

\begin{proof}[Proof of Lemma~\ref{lem:antipodal} for $k=3$]
  Let $S = \{s_1, s_2, s_3\}$ with $|s_1| \leq |s_2| \leq |s_3|$.
  Again assume the opposite to the statement of the lemma, that is,
  assume that $\diam(S) \leq \pi - \dsiii$ and that there is an
  antipodal pair for which the shortest path between the two points
  uses at least two shortcuts.

  We now show several properties of the
  configuration~$S$.
  \begin{enumerate}[(i)]
  \item $|s_1|< 1.45 < \muiii$: For two shortcuts to be a valid
    combination their combined length must be at most $\pi - \dsiii$,
    hence, $|s_1|\leq(\pi - \dsiii)/2 < 1.45 < \muiii \approx 1.5943$.
  \item $\umbra(s_2) \cap \umbra(s_3) = \emptyset$: Otherwise we
    have an antipodal pair $(p, q)$ with $p, q\in \umbra(s_2) \cap
    \umbra(s_3)$. Using Observation~\ref{obs:umbra} and $|s_1| <
    \muiii$ by~(i) gives $d_S(p, q) = d_{s_1}(p, q) > \pi - \dsiii$.
  \item $|s_2| < \pi/2 < \muiii$: From~(ii) we have $|s_2| + |s_3| <
    \pi$ and thus $|s_2| < \pi/2 < \muiii$.
  \item $\umbra(s_1) \cap \umbra(s_3) = \emptyset$: As
    in~(ii) this follows from $|s_2| < \muiii$.
  \item $\delta(s_1) < 0.09$: From $|s_1| < 1.45$ by~(i).
  \item $\delta(s_2) < 0.12$: From $|s_2| < \pi/2$ by~(iii).
  \item $|s_3| > 1.32$: Pick an antipodal pair $p, q$ with $p$
    in the deep umbra~$\deep(s_2)$.
    This is always possible by Lemma~\ref{lem:deep}.
    By Observations~\ref{obs:deep}
    and~\ref{obs:deltas}, we have $d_{S}(p,q) = d_{\{s_1, s_3\}}(p,q)
    \geq \pi - 2(\delta(s_1) + \delta(s_3))$. This implies
    $\delta(s_1) + \delta(s_3) \geq \dsiii/2$.  Therefore $\delta(s_3)
    \geq \dsiii/4 > 0.06$, implying that~$|s_3| > 1.32$.
  \end{enumerate}

  Now let $p$ and $q$ be the midpoints of the inner and outer umbra
  of~$s_3$.  Without loss of generality assume that $\seg{pq}$ is
  vertical with $p$ below $q$ as shown in Figure~\ref{fig:up2six}.
  Since the umbras of~$s_1, s_2$ are disjoint from $\umbra(s_3)$ and
  $\delta(s_1) \leq \delta(s_2) < 0.12 < |s_3|/2$ by~(vi) and~(vii),
  $s_1$ and $s_3$ do not cross~$\seg{pq}$ and lie either in the left
  or right semicircle determined by~$\seg{pq}$.

  Since $|s_1| \leq |s_2| < \muiii$ by~(i) and~(iii), we have
  $d_{s_1}(p, q) \geq d_{s_2}(p, q) > \pi - \dsiii$, and so the
  pair~$p, q$ must use $s_1$ and $s_2$ in combination.  But this means
  that $s_1$ and $s_2$ lie in the same semicircle of~$C$, let's say
  the left one as shown in Figure~\ref{fig:up2six}.

  Consider now the point~$q'=q - \dsiii + \eps$, for some small~$\eps
  > 0$. Traveling counter-clockwise from~$p$ to~$q'$ cannot use any
  shortcut and has length~$\pi - \dsiii + \eps > \pi - \dsiii$,
  traveling clockwise using~$s_1$ and~$s_2$ and arguing as in
  Observation~\ref{obs:deltas} has distance at least $\pi + \dsiii -
  \eps - 2\delta(s_1) - 2\delta(s_2) > \pi + 0.25 - 0.18 - 0.24 - \eps
  > \pi - 0.25 > \pi - \dsiii$, which is a contradiction.
\end{proof}

\begin{figure}[hbt]
  \centerline{\includegraphics{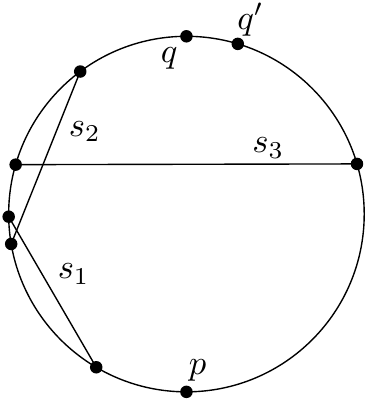}}
  \caption{Proof of Lemma~\ref{lem:antipodal} for~$k{=}3$.}
  \label{fig:up2six}
\end{figure}

In order to handle the remaining case $k \in \{4,5,6\}$ we need a
bound on the lengths of shortcuts that appear in combination.
\begin{lemma}
  \label{lem:shortlong} Let $S' \subset S$ be the set of shortcuts
  used by the shortest path for an antipodal pair $(p, q)$.  If
  $d_{S'}(p,q) \leq \pi - \dsk$, where $k \in \{4, 5, 6\}$, then the
  longest shortcut in $S'$ has length at least~$\longk$, all others
  have total length at most $\shortk$.  Here $\shortk$ and $\longk$,
  with $\shortk < \longk$, are the two solutions to the equation
  $\delta(x) + \delta(\pi - \dsk - x) = \dsk/2$ for $x \in [\pi - \dsk
  - 2, \,2]$.
\end{lemma}
\begin{proof}
  The function $x \mapsto \delta(x)$ is increasing and strictly convex
  on the interval $[0,2]$.  Therefore the function $x \mapsto g(x)
  = \delta(x) + \delta(\pi - \dsk - x)$ is strictly convex on the
  interval $[\pi - \dsk - 2, 2]$, it is also symmetric about~$x_{0} =
  (\pi - \dsk)/2$. Since $g(x_0) < \dsk/2 < g(2)$, there are thus
  exactly two solutions $\shortk$ and $\longk$ to the equation~$g(x)
  = \dsk/2$. Because of the symmetry of~$g(x)$ we have $\shortk
  + \longk = \pi - \dsk$. We computed $\shortk$ and~$\longk$
  numerically and list their values in \tablename~\ref{table:values}.

  \begin{table}[hbt]
    \centerline{\begin{tabular}{c||cc|c|cc}
        $k$ & $\ask$ & $\dsk$ & $\muk$ & $\shortk$ & $\longk$ \\
        \hline
        4 & 1.9619 & 0.3943 & 1.7623 & 1.0373 & 1.7100 \\
        5 & 1.9969 & 0.5164 & 1.8526 & 0.7862 & 1.8390 \\
        6 & 2.0000 & 0.5707 & 1.8828 & 0.6958 & 1.8751
    \end{tabular}}
    \caption{Numeric values for $\shortk$, and $\longk$.}
    \label{table:values}
  \end{table}
  
  Let $S' = \{s_1, \dots, s_n\}$ with $|s_1| \geq |s_2| \geq \dots
  \geq |s_n|$.  If $n = 1$ the statement follows from $\muk > \longk$,
  so assume $n \geq 2$.

  We use Karamata's theorem.  It states that if $f$~is a strictly
  convex non-decreasing function on an interval~$I$ and $x_{1} \geq
  x_{2} \geq \dots \geq x_{n}$ and $y_{1} \geq y_2 \geq \dots y_n$ are
  values in~$I$ such that $x_1 + \dots + x_i \geq y_1 + \dots + y_i$
  for all $1 \leq i \leq n$, then
  $\sum_{i=1}^{n}f(x_i) \geq \sum_{i=1}^{n}f(y_i)$, and equality holds
  only if $x_{i} = y_{i}$ for all~$i$.

  We apply this with $f(x) = \delta(x)$ on the interval~$I = [0, 2]$.
  We set $y_{i} = |s_i|$, $x_{1} = \longk$, $x_2 = \shortk$, and $x_i
  = 0$ for $i > 2$.  Assume for a contradiction that $y_1 = |s_1|
  < \longk$.  Then the conditions of Karamata's theorem are satisfied:
  $x_1 = \longk > y_1$ and $x_1 + x_2 = \longk + \shortk = \pi
  - \dsk \geq d_{S'}(p,q) \geq \sum_{i=1}^{n}|s_i| = \sum_{i=1}^{n}
  y_i$.  We thus have $\sum_{i=1}^{n}\delta(s_{i}) < \delta(\longk)
  + \delta(\shortk) = \dsk/2$.  By Observation~\ref{obs:deltas} we
  then have $d_{S'}(p,q) \geq \pi - 2\sum_{i=1}^{n}\delta(s_i) > \pi
  - \dsk$, a contradiction.

  It follows that~$|s_1| \geq \longk$. Then $\sum_{i=1}^{n}
  |s_i| \leq \pi - \dsk$ implies $\sum_{i=2}^{n}|s_i| \leq \pi - \dsk
  - \longk = \shortk$.
\end{proof}

\begin{proof}[Proof of Lemma~\ref{lem:antipodal} for $k\in \{4,
    5, 6\}$]
  Let $S = \{s_1, s_2, \dots, s_k\}$ with $|s_1| \leq |s_2| \leq
  \dots \leq |s_k|$, and assume that some antipodal pair of points
  needs to use more than one shortcut.  By Lemma~\ref{lem:shortlong} this
  means that~$|s_1| \leq \shortk$.

  Consider an antipodal pair $p,q$ with $p \in \deep(s_k)$. By
  Observation~\ref{obs:deep} it cannot use~$s_k$, and needs either a
  single shortcut of length at least~$\muk$, or a shortcut combination
  whose longest element has length at least~$\longk$ by
  Lemma~\ref{lem:shortlong}.  This implies that $|s_{k-1}| \geq
  \longk$.

  Since $|s_{k-1}| + |s_k| \geq 2\longk > \pi + 0.2$, the intersection
  $\umbra(s_{k-1}) \cap \umbra(s_k)$ has total length at least~$0.4$.
  Since it consists of at most four arcs, one arc has length at
  least~$0.1$.  Let $p$ be the midpoint of this intersection arc, and
  $q$ be its antipode.  Since the interval $[p-0.05, p+0.05] \subset
  U(s_{k-1}) \cap U(s_k)$, Lemma~\ref{lem:deep} implies that either
  $p$ or $q$ lies in the deep umbra~$\deep(s_{k-1})$, and either $p$
  or $q$ lies in~$\deep(s_k)$.  By Observation~\ref{obs:deep} the pair
  $(p, q)$ cannot use either $s_{k-1}$ or $s_k$.  As above we can now
  conclude that~$|s_{k-2}| \geq \longk$.

  Let us call a shortcut~$s_i$ \emph{short} if $|s_i| \leq \shortk$,
  and let $\ell$ be the number of short shortcuts.  By the above, $1
  \leq \ell \leq k - 3$.  Let $S' = \{s_{\ell+1}, \dots, s_k\}$ be the
  set of~$k-\ell$ shortcuts that are not short.

  We claim that $\diam(S') \leq \pi - \dsk + 2\ell\delta(\shortk)$.
  Indeed, for any pair of points $p, q \in C$ there is a path from~$p$
  to~$q$ using~$S$ of length at most~$\pi - \dsk$.  Arguing as in
  Observation~\ref{obs:deltas}, we replace a short shortcut~$s_i$ by
  walking along the circle, and obtain a path~$\gamma$ of length at
  most~$\pi - \dsk + \ell\times2\delta(\shortk)$.  By
  Lemma~\ref{lem:shortlong}, if $p, q$ is an antipodal pair,
  then~$\gamma$ uses only one shortcut of~$S'$.

  Set $\dht := \dsk - 2(k-3)\delta(\shortk)$. Since $\ell \leq k - 3$,
  we have $\diam(S') \leq \pi - \dht$.  We consider the
  strip~$\strip(\dht)$.  The middle line $\midline$ of~$\strip(\dht)$
  corresponds to antipodal pairs~$p, q \in C$. By the argument above,
  there is $s \in S'$ such that $\dists(p, q) \leq \pi - \dht$.  It
  follows that the regions~$\reg(s, \dht)$, for $s \in S'$ and $|s|
  \geq \longk$, cover~$\midline$ entirely.  The width of such a region
  is at most $\pi - \longk - \dht$, and so we must have $m \cdot (\pi
  - \longk - \dht) \geq \pi$, where $m$ is the number of shortcuts
  in~$S'$ of length at least~$\longk$.  Calculation shows that $m \geq
  k-1$.  Since $m \leq |S'| = k - \ell \leq k - 1$, this implies $\ell
  = 1$, and $|s_2| \geq \longk$.

  Since $2 \cdot \longk > \pi$, no two shortcuts in $S'$ can be
  combined, so $\strip(\dht)$ must be entirely covered by the
  regions~$\reg(s_i, \dht)$, for $i \in \{2, 3,\dots,k\}$.  The area
  of~$\strip(\dht)$ is $4\dht\pi$. By Lemma~\ref{lem:area-max}, the
  area of~$\reg(s_i, \dht)$ is at most $A(\aht, \dht) = 4\dht(\pi -
  \aht - \dht)$, where $\aht$ is such that $\delta(\aht) = \dht$, and
  so we must have $(k-1) \times (\pi - \aht - \dht) \geq \pi$.
  However, calculation shows that this is false, contradicting our
  assumption that some pair of antipodal points uses more than one
  shortcut.
\end{proof}


\subsection{Optimality of our configurations}
\label{ssec:optimality}

It remains to show that the configurations of
\figurename~\ref{fig:opt_upto5} are optimal even if combinations of
shortcuts can be used.  We prove this in the following lemma:
\begin{lemma}
  \label{lem:upto6-combined}
  Let $S$ be a set of $k$ shortcuts for $k \in \{3, 4, 5\}$ such
  that $\diam(S) \leq \pi - \dsk$.  Then there is no pair of points
  $p, q \in C$ such that the path of length~$d_{S}(p, q)$ uses more
  than one shortcut.
\end{lemma}
\begin{proof}
  By Lemma~\ref{lem:antipodal} pairs of antipodal points cannot use
  more than one shortcut.  This implies that the middle line
  $\midline$ of $\strip(\dsk)$ is covered by the regions
  $\reg(s_i,\dsk)$.  The region $\reg(s_i,\dsk)$ intersects $\midline$
  only if $|s_i| \geq \muk$, so by Corollary~\ref{coro:region}
  $\reg(s_i,\dsk)$ covers at most $2(\pi - \muk - \dsk)$
  of~$\midline$.  Calculation shows that $(k-1)(\pi - \muk - \dsk) <
  \pi$, so all $k$~shortcuts have length at least~$\muk$.
  Since $2\muk > \pi$, this implies that no shortcuts can be
  combined.
\end{proof}

Combining
Lemmas~\ref{lem:optimal-no-combination},~\ref{lem:two-optimal},
and~\ref{lem:upto6-combined}, we obtain our first theorem.
\begin{theorem}
  \label{thm:2to5}
  For $k \in \{2, 3, 4, 5\}$ there is a set~$S$ of~$k$ shortcuts that
  achieves $\diam(S) = \pi - \dsk$. This is optimal and the solution
  is unique up to rotation.
\end{theorem}


\section{Six and seven shortcuts}

The configuration of six shortcuts of length~$2$ (that is, all
shortcuts are diameters of the circle) shown in
\figurename~\ref{fig:opt_6} achieves diameter~$\pi - \delta(2) = \pii
+ 1$. Unlike the cases $2 \leq k \leq 5$, this configuration is not
unique---it can be perturbed quite a bit without changing the
diameter.
\begin{figure}[hbt]
  \centerline{\includegraphics{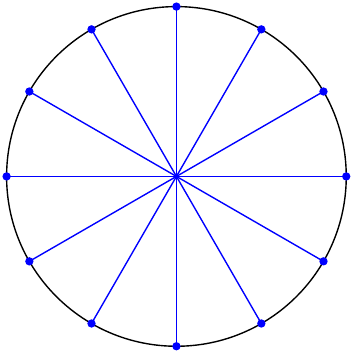}}
  \caption{An optimal configuration of six shortcuts. }
  \label{fig:opt_6}
\end{figure}

It remains to argue that the configuration is indeed optimal, that is,
there is no set $S$ of six shortcuts that achieves $\diam(S) < \pi -
\delta(2)$.  Here, we cannot use a simple area argument as in the case
$k < 6$, as the regions of the optimal solution in $\strip(\dsvi)$
overlap heavily.

In fact, we can show that even if we allow \emph{seven} shortcuts,
there is no set $S$ of shortcuts that achieves $\diam(S) < \pi -
\delta(2)$.  This implies a collapse between the cases of $k=6$ and
$k=7$, that is, $\diam(7) = \diam(6) = \pi - \delta(2)$.
\begin{theorem}
  \label{thm:6to7}
  There is a set~$S$ of six shortcuts that achieves $\diam(S) = \pi -
  \delta(2) = \pii + 1$.  There is no configuration of six or seven
  shortcuts that has diameter smaller than~$\pii+1$.
  Therefore, we have $\diam(7) = \diam(6) = \pii+1$.
\end{theorem}

\subsection{A short proof for six shortcuts\ldots}

Since the proof for seven shortcuts is quite long and rather
technical, we first give a short proof for the case of six shortcuts
(even though this is of course implied by the proof for seven
shortcuts).
\begin{proof}[Proof of Theorem~\ref{thm:6to7} for six shortcuts]
  Let $S = \{s_1, \dots, s_6\}$ with $|s_1| \leq \dots \leq |s_6|$,
  and assume that $\diam(S) \leq \pi - \ds$ with $\ds > \dsvi$.

  By Lemma~\ref{lem:antipodal}, this implies that no antipodal pair
  uses more than one shortcut.  This means that the middle
  line~$\midline$ of $\strip(\ds)$ is covered by the
  regions~$\reg(s_{i}, \ds)$.  The region $\reg(s_i,\ds)$
  intersects $\midline$ only if $|s_i| \geq \muvi$, so by
  Corollary~\ref{coro:region} $\reg(s_i,\ds)$ covers at most $2(\pi
  - \muvi - \dsvi)$ of~$\midline$.  Calculation shows that $4(\pi -
  \muvi - \dsvi) < \pi$, so at least five shortcuts have length at
  least~$\muvi$, that is $|s_2| \geq \muvi$.

  If no pair of points uses more than one shortcut, then the
  strip~$\strip(\ds)$ must be covered by the six regions~$\reg(s_i,
  \ds)$. In particular, the upper boundary~$\boundary$ of
  $\strip(\ds)$ is covered.  Since $\ds > \dsvi = \delta(2) \geq
  \delta(a)$ for $a \in [0,2]$, by Corollary~\ref{coro:region}
  region~$\reg(s_{i}, \ds)$ covers at most~$\pi - |s_i| - \ds$
  of~$\boundary$.  This leads to a contradiction:
  \[
  \sum_{i=1}^{6}(\pi - |s_i| - \ds) \leq (\pi - \dsvi) + 5(\pi
  - \muvi - \dsvi) = 6\pi - 6\dsvi - 5\muvi \approx 6.0106 < 2\pi.
  \]

  It follows that some pair of points uses two shortcuts.  Since
  $2\muvi > \pi$, this can only be a combination involving~$s_1$, so
  we must have $|s_1| + |s_2| \leq \pi - \ds$. From $|s_2| \geq \muvi$
  we get $|s_1| \leq \pi - \ds - \muvi$, and so $\delta(s_1) \leq
  \delta(\pi - \dsvi - \muvi) < 0.008$.  We set $S' = \{s_2, \dots,
  s_6\}$, and have $\diam(S') \leq \pi - \dht$, where $\dht = \dsvi -
  0.016$ by Observation~\ref{obs:deltas}.  Since no two shortcuts
  in~$S'$ can be combined, the strip~$\strip(\dht)$ of area $4\dht\pi$
  is covered by the regions~$\reg(s_i, \dht)$, for $2 \leq i \leq 6$.
  By Lemma~\ref{lem:area-max}, the area of~$\reg(s_i, \dht)$ is at
  most $A(\aht, \dht) = 4\dht(\pi - \aht - \dht)$, where $\aht$ is
  such that $\delta(\aht) = \dht$. However $5 (\pi - \aht - \dht)
  \approx 2.9353 < \pi$, another contradiction.
\end{proof}


\subsection{\ldots and a long proof for seven}

We let $S = \{s_1, s_2, \dots, s_7\}$ with $|s_1| \leq |s_2| \leq
\dots \leq |s_7|$, set $\ds = \dsvi + \eps = \delta(2) + \eps$ for some
small $\eps > 0$, and assume that $\diam(S) \leq \pi - \ds$.  We will
show that this leads to a contradiction.

We will use the following short result.
\begin{lemma}
  \label{lem:longvi}
  A region $\reg(s, \dht)$ reaches the middle line of~$\strip(\dht)$
  for $\dht = \delta(2) - 2\delta(\shortvi)$ if and only if~$|s| \geq
  \longvi$.
\end{lemma}
\begin{proof}
  Recall from Lemma~\ref{lem:shortlong} that $\shortvi$ and $\longvi$
  satisfy $\delta(\shortvi) + \delta(\longvi) = \delta(2)/2$. Region
  $\reg(s, \dht)$ reaches the middle line if and only if $2\delta(s) \geq \dht =
  \delta(2) - 2\delta(\shortvi)$ by Corollary~\ref{coro:region}
  or, equivalently, $\delta(s) \geq \delta(2)/2 - \delta(\shortvi)$.
  This is equivalent to~$|s| \geq \longvi$.
\end{proof}

We start as in the proof of Lemma~\ref{lem:antipodal} for $k \in
\{4,5,6\}$ and argue that at least three of the shortcuts in~$S$ have
length at least~$\longvi$, that is, $|s_5| \geq \longvi$.  Let $\ell$
denote the number of short shortcuts (of length at most~$\shortvi$).
We have $0 \leq \ell \leq 4$, and we let $S' = \{s_{\ell+1}, \ldots,
s_7\}$ be the set of $7-\ell$ shortcuts that are not short. Observe
that~$\diam(S') \leq \pi - \dht(\ell)$, where $\dht(\ell) = \ds -
\ell\cdot2\delta(\shortvi)$.  Again, we are here using the same
argument as in the proof of Lemma~\ref{lem:antipodal} for $k \in
\{4,5,6\}$.
\begin{lemma}
  \label{lem:Sp6}
  We have $|S'| \geq 6$, $\diam(S') \leq \pi - \dht(1) = \pi -
  (\ds-2\delta(\shortvi))$, and~$|s_3| \geq \longvi$.
\end{lemma}
\begin{proof}
   Consider the strip $\strip(\dht(4))$ and its middle
   line~$\midline$. A point on the middle line corresponds to an
   antipodal pair, and by Lemma~\ref{lem:shortlong} the shortest path
   for an antipodal pair can use at most one shortcut of length larger
   than~$\shortvi$.  It follows that the regions $\reg(s, \dht(4))$
   for $s \in S'$ must cover~$\midline$.  The region $\reg(s,\dht(4))$
   only reaches~$\midline$ if $2\delta(s) \geq \dht(4)$, which implies
   $|s| > 1.849$.  By Corollary~\ref{coro:region}, the width of the
   two rectangles of such a region is at most $\pi - 1.849 -
   \dht(4)$. Since $4 \times (\pi - 1.849 - \dht(4)) < \pi$, there
   must be at least five shortcuts of length at least~$1.849 >
   \shortvi$, and so~$|S'| \geq 5$ and therefore $\ell \leq 2$.  This
   implies that $\diam(S') \leq \pi - \dht(2) < \pi - \ds_5$.
   Theorem~\ref{thm:2to5} now implies~$|S'| \geq 6$.  This in turn
   means $\ell \leq 1$ and therefore $\diam(S') \leq \pi -
   \dht(1)$. We now redo the argument above: The region
   $\reg(s,\dht(1))$ only reaches~$\midline$ if $s \geq \longvi$ by
   Lemma~\ref{lem:longvi}.  The width of the two rectangles of such a
   region is at most $\pi - \longvi - \dht(1)$. Since $4 \times (\pi -
   \longvi - \dht(1)) < \pi$, there must be at least five shortcuts of
   length at least $\longvi$, that is, $|s_3| \geq \longvi$.
\end{proof}

We will need the following lemma about the \emph{six} shortcuts $s_2,
s_3, \dots, s_7$:
\begin{lemma}
  \label{lem:s2dht}
  If $|s_2| > 1.7$ and $\diam(\{s_2,s_3,\dots,s_7\}) \leq \pi - 0.54$,
  then $|s_2| > 1.999$.
\end{lemma}
\begin{proof}
  Since $|s_2| + |s_3| > 1.7 + \longvi > \pi$, no combinations of
  the shortcuts $s_2, s_3, \dots, s_7$ are possible, and so the six
  regions $\reg(s_2, 0.54),\dots,\reg(s_7, 0.54)$ must cover the upper
  boundary~$\boundary$ of~$\strip(0.54)$.  Let $\aht$ be such that
  $\delta(\aht) = 0.54$.  Since $\delta(1.999) < 0.54$, we have $\aht
  > 1.999$. Lemma~\ref{lem:region} and Corollary~\ref{coro:region}
  imply the following: If $|s_i| \geq \aht$, then $\reg(s_i, 0.54)$
  covers \emph{two} segments of $\boundary$ of length at most~$\pi -
  1.999 - 0.54 < 0.603$; if $\longvi \leq |s_i| < \aht$, then
  $\reg(s_i, 0.54)$ covers \emph{one} segment of $\boundary$ of length
  at most~$\pi - \longvi - 0.54 < 0.727$; and if $1.7 \leq |s_i| <
  \longvi$, then $\reg(s_i, 0.54)$ covers \emph{one} segment of
  $\boundary$ of length at most~$\pi - 1.7 - 0.54 < 0.902$.

  Assume that $|s_4| < \aht$.  Then the coverage of the upper
  boundary~$\boundary$ of~$\strip(0.54)$ by the six regions is at most
  $0.902 + 2 \times 0.727 + 3 \times 2 \times 0.603 < 2\pi$, a
  contradiction.  So we have~$|s_4| \geq \aht$.

  Therefore the five regions $\reg(s_3, 0.54),\dots,\reg(s_7, 0.54)$
  consist of ten rectangles of total width at most $2 \times 0.727 + 8
  \times 0.603 < 2\pi$. This implies that there must be a~$\theta$
  such that the segment $\{(\theta, \xi)\mid -0.54 \leq \xi \leq
  0.54\}$ is disjoint from these five regions.  The segment must
  therefore be contained in $\reg(s_2, 0.54)$. This is only possible
  if $\reg(s_2, 0.54)$ covers the entire height of~$\strip(0.54)$, or
  equivalently, if $\delta(s_2) \geq 0.54$ by Lemma~\ref{lem:region}.
  This implies that~$|s_2| \geq \aht > 1.999$.
\end{proof}

We now distinguish two cases, based on the length of~$s_1$.


\subsubsection{A short shortcut exists}
\label{sssec:short}

We first assume that $|s_1| \leq \shortvi$, so that~$|S'| = 6$.
By Lemma~\ref{lem:Sp6} we have $\diam(S') \leq \pi - \dht$, where
$\dht = \dht(1) = \ds - 2\delta(\shortvi)$.
\begin{lemma}
  \label{lem:s2longvi}
  $|s_2| \geq \longvi$.
\end{lemma}
\begin{proof}
  Assume for a contradiction that $|s_2| < \longvi$. By
  Lemma~\ref{lem:longvi}, $\reg(s_2, \dht)$ does not reach the middle
  line~$\midline$ of strip~$\strip(\dht)$, and so the regions of the
  remaining five shortcuts~$s_3, \ldots , s_7$ in~$S'$ must
  cover~$\midline$.  By Corollary~\ref{coro:region}
  we have
  \[
  \sum_{i = 3}^{7} 2(\pi - |s_i| - \dht) \geq 2\pi
  \quad \text{or, equivalently,} \quad
  \sum_{i=3}^{7} |s_i| \leq 4\pi - 5\dht,
  \]
  which implies that the shortest of these five segments has length
  $|s_3| \leq (4\pi - 5\dht)/5 = \frac 45 \pi - \dht$.

  On the other hand, the six regions $\reg(s_2, \dht), \ldots ,
  \reg(s_7, \dht)$ must cover~$\strip(\dht)$ entirely.  The
  strip~$\strip(\dht)$ has area~$4\dht\pi \approx 6.9862$. By
  Lemma~\ref{lem:area-max} and using~$\delta(1.999) < \dht$, the total
  area of the six regions is
  \begin{align*}
  \sum_{i=2}^{7} A(|s_i|, \dht)
  & = A(|s_2|, \dht) + A(|s_3|, \dht) + \sum_{i=4}^{7} A(|s_i|, \dht) \\
  & \leq
  A(\longvi,\dht) +
  A(\frac 45 \pi - \dht, \dht) + 4 A(1.999, \dht)
  \approx 6.9765 < 6.9862 \approx 4 \dht\pi,
  \end{align*}
  a contradiction.
\end{proof}
Since $\longvi > 1.7$, we can now apply Lemma~\ref{lem:s2dht}, and
obtain $|s_2| > 1.999$, implying that \emph{all} six shortcuts in~$S'$
have length larger than~$1.999$.

Finally we return to the full set~$S$ with diameter~$\diam(S) \leq \pi
- \ds$.  We rotate the configuration such that the midpoint of the
inner umbra of~$s_1$ is at coordinate zero.
\begin{lemma}
  \label{lem:thetaset}
  The pairs corresponding to configurations~$(\theta, \ds)$ on the
  upper boundary~$\boundary$ of~$\strip(\ds)$ with $\theta$ in the
  following set
  \[
  \big\{ \pi - \dsii \big\} \cup
  \big[ \pi + 0.4, 2\pi - 0.4 \big] \cup
  \big\{ \dsii \big\}
  \]
  cannot use shortcut~$s_1$.
\end{lemma}
\begin{proof}
  Recall that the point pair for the configuration~$(\theta, \ds)$
  consists of $p = \theta - \dsii$ and $q = \theta + \pi + \dsii$.  So
  for $\theta = \pi - \dsii$ we have $p = \pi - \ds$ and $q = 0$,
  while for $\theta = \dsii$ we have $p = 0$ and $q = \pi + \ds$.
  Since~$0$ lies in the deep umbra of~$s_1$,
  Observation~\ref{obs:deep} implies that $s_1$ cannot be used.

  Consider now $\theta \in [\pi + 0.4, 2\pi - 0.4]$. We have
  \begin{align*}
    \pi + 0.11 < \pi + 0.4 - \dsii & \leq p \leq
    2\pi - 0.4 - \dsii < 2\pi - 0.68\\
    0.68 < 0.4 + \dsii & \leq q \leq \pi - 0.4 + \dsii < \pi - 0.11
  \end{align*}
  Assume for a contradiction that there is a shortest path~$\gamma$
  from~$p$ to~$q$ that uses~$s_1$.  Since $s_1$ lies on the long arc
  from~$p$ to~$q$ of length~$\pi + \ds$, the path~$\gamma$ must also
  use another shortcut~$s_i$ with $|s_i| > 1.999$. On the other hand,
  $\gamma$ must use~$s_1$ and then go along the circle boundary
  to~$p$ or~$q$.  Since neither~$p$ nor~$q$ lie in the
  interval~$(-0.68, 0.68)$, this subpath of~$\gamma$ has length at least
  \[ |s_1| + 0.68 - |s_1|/2 - \delta(s_1) = 0.68 + |s_1|/2 - \delta(s_1),\]
  which is strictly larger than $0.68$.
  So the entire
  path~$\gamma$ has length at least $1.999 + 0.68 > \pi - \delta(2) > \pi - \ds$,
  a contradiction.
\end{proof}

It follows that the six regions~$\reg(s_2, \ds), \reg(s_3, \ds),
\dots, \reg(s_7,\ds)$ must cover the points on the upper boundary~$\boundary$
with~$\theta$ in the set above. Since the shortcuts have length at
least~$1.999$ and $\ds > \delta(2)$, each region covers a single
interval on~$\boundary$ of length at most~$\pi - 1.999 - \ds < 0.572$.

The interval $[\pi + 0.4, 2\pi - 0.4]$ has length~$\pi - 2 \times 0.4
> 4 \times 0.572$, and therefore requires five regions to be covered.
The distance between the isolated point~$\pi - \dsii$ and the interval
is~$0.4 + \dsii > 0.68$, and the same holds for the distance
between~$\dsii$ and the interval. Both isolated points thus require a
region that covers them and that cannot contribute to the coverage of
the interval. It follows that we need seven regions to cover this
subset of~$\boundary$, a contradiction.


\subsubsection{No short shortcut}
\label{sssec:no-short}

We now assume that $|s_1| > \shortvi$, that is $|S'| = 7$.  Since
$|s_3| \geq \longvi$ and $|s_1|+|s_3| > \shortvi + \longvi = \pi - \dsvi = \pi -
\delta(2) > \pi - \ds$, the only possible combination of shortcuts
that can be used is the combination of $s_1$ and~$s_2$.

\begin{lemma}
  \label{lem:bound-combining-edges}
  $\delta(s_1) + \delta(s_2) < 0.2$
\end{lemma}
\begin{proof}
  We first assume that the combination of shortcuts $s_1$ and $s_2$ is
  never used. Then the seven regions~$\reg(s_i, \ds)$ for $i = 1, 2,
  \ldots, 7$ cover the upper boundary~$\boundary$ of $\strip(\ds)$.
  By Corollary~\ref{coro:region} we have
  \[
  \sum_{i=1}^{7} (\pi - |s_i| - \ds) \geq 2\pi
  \quad \text{or, equivalently,} \quad
  |s_1| + |s_2| \leq 5\pi - 7\ds - \sum_{i=3}^{7} |s_i|
  \]
  and using $|s_3| \geq \longvi$ gives
  \[
  |s_1| + |s_2| \leq 5\pi - 7\ds - 5\longvi < 2.34.
  \]
  Since $|s_1| \geq \shortvi$, convexity of $\delta$ gives us
  $\delta(s_1) + \delta(s_2) \leq \delta(\shortvi) + \delta(2.34 -
  \shortvi) \approx 0.15 < 0.2$, proving the claim.

  It remains to consider the case where for some pair of points the
  combination of $s_1$ and $s_2$ needs to be used (and one of the two
  shortcuts alone does not suffice).  This implies that $|s_1| + |s_2|
  \leq \pi-\ds < \pi - \delta(2) = \pii + 1$.

  If $|s_1| \geq 0.83$, then convexity of $\delta$ implies
  $\delta(s_1) + \delta(s_2) \leq \delta(0.83) + \delta(\pii + 1 -
  0.83) \approx 0.1986 < 0.2$.

  If $|s_1| < 0.83$ and $|s_2| \leq 1.7$, then $\delta(s_1) +
  \delta(s_2) \leq \delta(0.83) + \delta(1.7) \approx 0.1789 < 0.2$.

  Finally, if $|s_1| < 0.83$ and $|s_2| > 1.7$, then we observe that
  $\diam(\{s_2, \dots, s_7\}) \leq \pi - \ds + 2\delta(s_1) < \pi - \ds +
  2\delta(0.83) < \pi - 0.54$.  We can thus apply
  Lemma~\ref{lem:s2dht} and find that $|s_2| > 1.999$. Since $\shortvi
  + 1.999 > \pi - \delta(2) > \pi - \ds$, no combination of $s_1$ and $s_2$ is
  possible and we are back in the first case.
\end{proof}

Let $\zeta = \delta(2) - 0.4 = \pii - 1.4 \approx 0.1708$, and
consider the configurations~$(\theta, \zeta)$ in $\strip(\ds)$.  Let
$(p, q)$ be a pair of points on $C$ corresponding to some~$(\theta,
\zeta)$ in $\strip(\ds)$.  The shorter arc $\arc{pq}$ along $C$ has
length~$\pi - \zeta = \pi - (\pii - 1.4) = \pii + 1.4$.  Thus, any
path from $p$ to $q$ using~$s_1$, $s_2$, or their combination has
length at least $\pi-\zeta-2\delta(s_1)-2\delta(s_2) > \pii + 1$ since
$\delta(s_1) + \delta(s_2) < 0.2$ according to
Lemma~\ref{lem:bound-combining-edges}.  Since $\pi-\delta^* < \pi -
\delta(2) = \pii + 1$, it follows that the pairs corresponding
to~$(\theta, \zeta)$ cannot make any use of~$s_1$ and~$s_2$.  That is,
they can only be covered by the five regions~$\reg(s_3, \ds),\ldots,
\reg(s_7, \ds)$.  Since $|s_3| \geq \longvi$, these regions have width
at most $\pi - \longvi -\ds < 0.696$.  For a region $\reg(s, \ds)$ to
cover \emph{two} pieces of $\xi = \zeta$ in $\strip(\ds)$, we need
$2\delta(s) \geq \ds + \zeta$ by Lemma~\ref{lem:region}, which implies
that $|s| > 1.949$.  This in turn means that the width of the region
is at most $\pi - 1.949 - \ds < 0.622$.  If follows that the five
regions~$\reg(s_3, \ds),\ldots, \reg(s_7, \ds)$ can cover at most $5
\times 2 \times 0.622 < 2\pi$ of the line $\xi = \zeta$
in~$\strip(\ds)$, a contradiction.

This concludes the proof of Theorem~\ref{thm:6to7} for seven shortcuts.


\section{Eight shortcuts}
\label{sec:eight}

With eight shortcuts we can improve on the diameter, obtaining
$\diam(8) < \diam(7) = \diam(6)$.
Our construction $S$ consists of six long shortcuts with length
$a_1 \approx 1.999870869$ and two short ones with
length $a_2 \approx 0.988571799$,
placed as in \figurename~\ref{fig:ex_8}(left),
and achieves the diameter $\diam(S) \approx \pi - 0.5822245291
= 2.559368125 < \diam(6)$.

We obtained this construction by maximizing $\ds$ with constraints
$\pi - a_1 - \ds \geq \pi / 6$, $\pi - a_2 - \ds \geq \pi/2$, and
$\delta(a_1) + \delta(a_2) \geq \ds$.  We can thus cover $\strip(\ds)$
as seen in the diagram in \figurename~\ref{fig:ex_8}(right).  In
particular, we have $\pi - a_2 - \ds = \pi/2$ and $\delta(a_1) +
\delta(a_2) = \ds$, while we have a strict inequality $\pi - a_1 - \ds
> \pi / 6$ in our construction.  So, in the strip~$\strip(\ds)$, the
regions slightly overlap.
\begin{figure}[thb]
  \centerline{\includegraphics{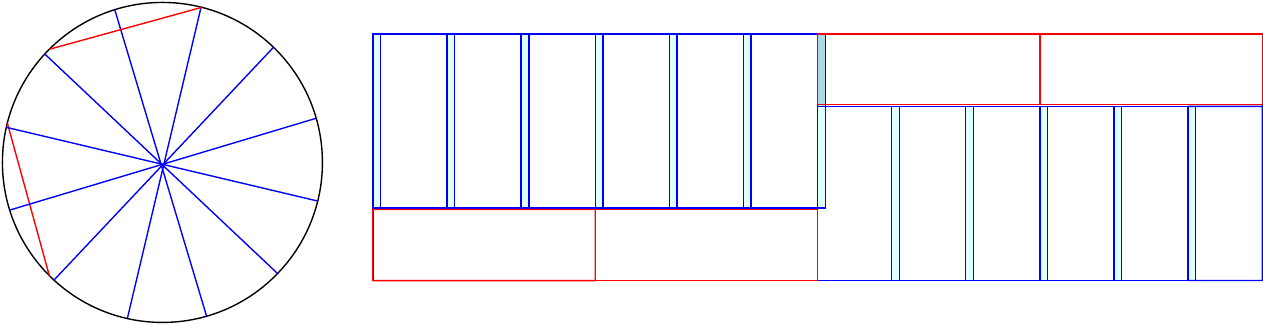}}
  \caption{A shortcut configuration $S$ of $8$ shortcuts with $\diam(S) <
    \diam(6)$, and the corresponding regions in the strip~$\strip(\ds)$.}
  \label{fig:ex_8}
\end{figure}


\section{An asymptotically tight bound}

In this final section, we show that $\diam(k) = 2 +
\Theta(1/k^{\nicefrac{2}{3}})$ as $k$ goes to infinity.

\begin{theorem}
  To achieve diameter at most~$2 + \nicefrac 1m$,
  $\Theta(m^{\nicefrac{3}{2}})$ shortcuts are both necessary and
  sufficient.
\end{theorem}
\begin{proof}
  We prove the necessary condition first.  Consider two points~$p, q$
  that form an angle of~$\pi - \nicefrac tm$, for some integer~$0 \leq
  t \leq \sqrt{m} - 2$.  Consider two intervals $I_{p}$ and $I_{q}$,
  both of arc length~$\nicefrac 4m$, with the midpoint of the
  intervals at~$p$ and~$q$, respectively. We claim that if there is no
  shortcut connecting a point of~$I_{p}$ with a point of~$I_{q}$, then
  the distance between~$p$ and~$q$ is larger than~$2 + \nicefrac 1m$.

  If there is no such shortcut, then the shortest path from~$p$ to~$q$
  must visit a point~$r$ on the circle not in either interval, see
  Figure~\ref{fig:asymptotic}(left).  The sum~$|pr| + |rq|$ is
  minimized when~$r$ is the point making angle~$\nicefrac 2m$
  with~$q$, so we have $\alpha(pr) = \pi - \nicefrac{(t+2)}{m}$ and
  $\alpha(rq) = \nicefrac 2m$.

  This gives us
  \begin{align*}
    |qr| &= 2\sin\frac{2}{2m} = 2 \sin\frac{1}{m} \geq
    \frac{2}{m} - \frac{2}{3!}\frac{1}{m^{3}} >
    \frac{2}{m} - \frac{1}{3m} = \frac{5}{3m}, \\
    |pr| &= 2\sin(\frac{\pi}{2} - \frac{t+2}{2m})
    = 2\cos\frac{t + 2}{2m} \geq 2\cos\frac{\sqrt{m}}{2m}
    = 2 \cos\frac{1}{2\sqrt{m}}
    \geq 2 - \frac{1}{4m},
  \end{align*}
  and so $|pr|+|rq| > 2 + \nicefrac 1m$.

  We now subdivide~$C$ into $\Theta(m)$ intervals of length at
  least~$\nicefrac 6m$.  Consider a pair of intervals~$I, J$ at arc
  distance at least $\pi - 1/\sqrt{m}$.  Then there are points~$p\in
  I$ and~$q\in J$ with $I_{p} \subset I$ and $I_{q} \subset J$ and $p,
  q$ forming an angle of the form~$\pi - \nicefrac tm$ for an integer
  $0 \leq t \leq \sqrt{m} - 2$.  It follows that there must be some
  shortcut connecting $I$ and~$J$.  Since there are
  $\Theta(m^{\nicefrac 32})$ such pairs of intervals, we must have at
  least $\Omega(m^{\nicefrac 32})$ shortcuts.
\begin{figure}[thb]
  \centerline{\includegraphics[width=.9\textwidth]{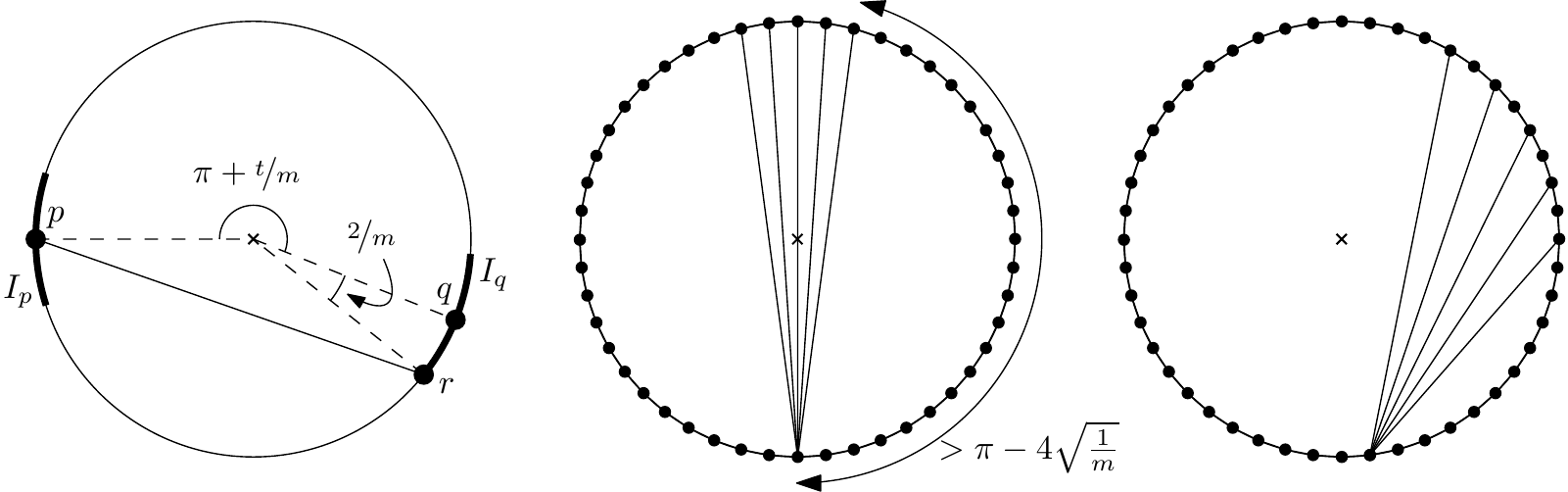}}
  \caption{(left) If there is no shortcut between $I_p$ and $I_q$ then
    the shortest path from~$p$ to~$q$ must visit a point~$r$ on the
    circle not in either interval. (center) Shortcut between every
    pair that makes an angle larger than $\pi - 4\sqrt{\nicefrac 1m}$.
    (right) Adding shortcuts of arc length~$\pi - \nicefrac tm$.  }
  \label{fig:asymptotic}
\end{figure}

  \medskip
  We now turn to the sufficient condition, and construct a set
  of~$\Theta(m^{\nicefrac 32})$ shortcuts that give a diameter
  of~$2+\nicefrac 1m$.

  We start by placing $4\pi m$ points uniformly around the circle, and
  connect each pair that makes an angle larger than $\pi -
  4\sqrt{\nicefrac 1m}$, as shown in
  Figure~\ref{fig:asymptotic}(center).  This creates
  $\Theta(m^{\nicefrac 32})$ shortcuts and ensures that for points $p,
  q$ with angle larger than $\pi - 4\sqrt{\nicefrac 1m}$ the distance
  between~$p$ and~$q$ is bounded by~$2 + \nicefrac 1m$.

  It remains to add shortcuts to decrease the distance of point
  pairs~$p, q$ that form an arc between~$2$ and~$\pi -
  4\sqrt{\nicefrac 1m}$.  For each integer $t$ with $4\sqrt{m} < t <
  2m$ we will create a set of shortcuts of arc length~$\pi - \nicefrac
  tm$, see Figure~\ref{fig:asymptotic}(right).  These shortcuts will
  be used for pairs~$p, q$ forming an arc between $\pi - \nicefrac tm$
  and $\pi - \nicefrac{(t-1)}{m}$.

  Let us fix such a~$t$, and consider a shortcut~$s$ of arc
  length~$\pi - \nicefrac tm$.  Then the length of the shortcut is
  \[
  |s| = 2\sin \frac{\pi - \nicefrac tm}{2} = 2 \cos\frac{t}{2m}.
  \]
  Using the bound $\cos x \leq 1 - \tfrac{x^{2}}{2} +
  \tfrac{x^{4}}{24} \leq 1 - (\tfrac 12 - \tfrac{1}{24})x^{2} = 1 -
  \tfrac{11}{24}x^{2}$ for $x < 1$, we have
  \[
  |s| \leq 2 - 2 \frac{11}{24} \frac{t^{2}}{4m^{2}}
  = 2 - \frac{11}{48}\frac{t^{2}}{m^{2}}
  < 2 - \frac{1}{6}\frac{t^{2}}{m^{2}} = 2 - 2\Delta,
  \]
  where we define $\Delta = \tfrac{1}{12} (\tfrac{t}{m})^{2}$.  Since
  $t > 4\sqrt{m}$ we have $\Delta > \tfrac{16}{12} \tfrac{1}{m} >
  \tfrac{1}{m}$.

  We repeat shortcuts of this length every arc interval of
  length~$\Delta$.  Consider now a pair of points~$p, q$ forming an
  angle in the interval~$\pi - \nicefrac tm$ to $\pi -
  \nicefrac{(t-1)}{m}$.  We can go from $p$ to~$q$ by first going to
  the nearest shortcut along an arc of length at most~$\Delta$, then
  following the shortcut of length at most~$2 - 2\Delta$, and finally
  going backwards by at most~$\Delta$, or forward by at
  most~$\nicefrac 1m < \Delta$. It follows that the distance
  between~$p$ and~$q$ is at most $2 - 2\Delta + 2\Delta = 2$.

  The number of shortcuts of length~$\pi - \nicefrac tm$ is
  $\nicefrac{2\pi}{\Delta}$, and so the total number of shortcuts of
  this type is
  \[
  \sum_{t = 4\sqrt{m} + 1}^{2m}24\pi\frac{m^{2}}{t^{2}} =
  24\pi m^{2}\sum_{t = 4\sqrt{m} + 1}^{2m} \frac{1}{t^{2}}
  \leq 24\pi m^{2} \int_{4\sqrt{m}}^{\infty} \frac{1}{x^{2}}\,dx
  = 6\pi m^{\nicefrac 32}.
  \]
  This completes the proof.
\end{proof}

\section{Conclusions}
We have given exact bounds on the diameter for up to seven shortcuts.
In all cases, the shortcuts are of equal length.  For $k=8$, however,
our upper bound construction uses shortcuts of two different lengths.
On the other hand,
it is not difficult to see that eight shortcuts of equal length
cannot even achieve a slightly better diameter than $\diam(6)$.
 In general, what is the diameter achievable with $k$
shortcuts of equal length?

We have shown that for $k = 0$ and $k = 6$ we have $\diam(k) =
\diam(k+1)$.  Are there any other values of $k$ for which this holds?

Finally, in all our constructions, including the one for large~$k$, we
never use combinations of shortcuts:  the shortest path for any pair of points
 uses at most one shortcut.  Is it true that combinations of
shortcuts never help, for any~$k$?   Meanwhile, one could
 make it a requirement and ask:  What is the best diameter
 achievable with $k$ shortcuts, under the restriction that no path can
use more than one shortcut?  This problem then reduces to covering the
strip $\strip(\ds)$ by regions $\reg(s, \ds)$, and may be more tractable
than the general form.


\newpage 
\appendix
\section{Appendix: Calculations}
\label{sec:calculations}
\input{calculations.tex}
\end{document}

%% file: calculations.tex
\begin{verbatim}
Source code at: http://github.com/otfried/circle-shortcuts

Lemma 4:
=========

delta(delta(2)) = 0.00402

Table 1:
========

k  & a*     & d*     & pi - d*& mu
2  & 1.4782 & 0.0926 & 3.0490 & 1.2219 \\
3  & 1.8435 & 0.2509 & 2.8907 & 1.5943 \\
4  & 1.9619 & 0.3943 & 2.7473 & 1.7623 \\
5  & 1.9969 & 0.5164 & 2.6252 & 1.8526 \\
6  & 2.0000 & 0.5708 & 2.5708 & 1.8828 \\

Lemma 10 for k = 3:
===================

a* = 1.8435, d* = 0.2509, mu = 1.5943
(i) (pi - d*)/2 = 1.4454
(v) delta(1.45) = 0.0860
(vi) delta(pi/2) = 0.1179
(vii) a such that delta(a) = 0.06 = 1.3150

Table 2 (in Lemma 11):
======================

k & a*     & d*     & mu_k   & sigma_k & lambda_k
4 & 1.9619 & 0.3943 & 1.7623 & 1.0373  & 1.7100 \\
5 & 1.9969 & 0.5164 & 1.8526 & 0.7862  & 1.8390 \\
6 & 2.0000 & 0.5708 & 1.8828 & 0.6957  & 1.8751 \\

Lemma 10 for k in {4, 5, 6}:
============================

Showing that l = 1:
k = 4:
  d* = 0.3943, sigma = 1.0373, lambda = 1.7100
  delta^ = 0.3411,  w = (pi - lambda - d^) = 1.0906
  (k-2) w = 2.1811 < pi
k = 5:
  d* = 0.5164, sigma = 0.7862, lambda = 1.8390
  delta^ = 0.4728,  w = (pi - lambda - d^) = 0.8298
  (k-2) w = 2.4894 < pi
k = 6:
  d* = 0.5708, sigma = 0.6957, lambda = 1.8751
  delta^ = 0.5262,  w = (pi - lambda - d^) = 0.7403
  (k-2) w = 2.9610 < pi

The final contradiction of Lemma 10:
k = 4:
  delta^ = 0.3411, a^ = 1.9304, w = pi - a^ - delta^ = 0.8701
  (k-1) w = 2.6103 < pi
k = 5:
  delta^ = 0.4728, a^ = 1.9893, w = pi - a^ - delta^ = 0.6795
  (k-1) w = 2.7178 < pi
k = 6:
  delta^ = 0.5262, a^ = 1.9979, w = pi - a^ - delta^ = 0.6174
  (k-1) w = 3.0872 < pi

Lemma 12:
=========

k=3: mu=1.5943, w = pi - mu - d* = 1.2964 => (k-1)w = 2.5928 < pi
k=4: mu=1.7623, w = pi - mu - d* = 0.9850 => (k-1)w = 2.9549 < pi
k=5: mu=1.8526, w = pi - mu - d* = 0.7726 => (k-1)w = 3.0902 < pi

Theorem 14 for k = 6:
=====================

mu=1.8828, d*=0.5708, w = pi - mu - d* = 0.6880
 4 w = 2.7518 < pi
 pi - d* + 5 w = 6.0106 < 2pi
 delta(pi - d* - mu) = 0.0072
 d^ = 0.5548, a^ = 1.9997
 5 (pi - a^ - d^) = 2.9353 < pi

Lemma 16:
=========

d* = 0.5708, sigma6 = 0.6957, delta(sigma6) = 0.0074
d^(4) = 0.5114, 2 delta(1.849) = 0.5104
  w = pi - 1.849 - d^(4) = 0.7812
  4 * w = 3.1248 < pi
d^(2) = 0.5411 > 0.5164 = d5*
d^(1) = 0.5559
  w = pi - lambda6 - d^(1) = 0.7105
  4 * w = 2.8422 < pi

Lemma 17:
=========

1.7 + lambda6 = 3.5751 > pi
delta(1.999) = 0.5397 < 0.54
  w1 = pi - 1.999 - 0.54 = 0.6026
  w2 = pi - lambda6 - 0.54 = 0.7265
  w3 = pi - 1.7 - 0.54 = 0.9016
  w3 + 2 * w2 + 6 * w1 = 5.9701 < 2pi
  2 * w2 + 8 * w1 = 6.2737 < 2pi

Lemma 18:
=========

d^ = d* - 2 * delta(sigma6) = 0.5559
s3 <= 0.8 * pi - d^ = 1.9573
A(lambda6, d^) = 4 * delta(lambda6) * (pi - lambda6 - d^) = 0.7900
A(s3, d^) <= 4 * delta(1.9573) * (pi - 1.9573 - d^) = 0.9681
delta(1.999) = 0.53967 < d^
A(1.999, d^) < 4 * d^ * (pi - 1.999 - d^) = 1.3046
0.7900 + 0.9681 + 4 * 1.3046 = 6.9765 < 6.9862 = 4 * d^ * pi

Lemma 19:
=========

0.4 - d*/2 = 0.1146, 0.4 + d*/2 = 0.6854

Final contradiction of Section 4.2.1 (a short shortcut exists):
===============================================================

pi - 1.999 - d* = 0.5718
pi - 2 * 0.4 = 2.3416 > 2.2872 = 4 * 0.5718
0.4 + d*/2 = 0.6854 > 0.5718

Lemma 20:
=========

s1 + s2 <= 5 * pi - 7 * ds - 5 * lambda6 = 2.3369 < 2.34
delta(sigma6) + delta(2.34 - sigma6) = 0.1505 < 0.2
delta(0.83) + delta(pi/2 + 1 - 0.83) = 0.1986 < 0.2
delta(0.83) + delta(1.7) = 0.1789 < 0.2
1.999 + sigma6 = 2.6947 > 2.5708 = pi - d*

Final contradiction of Section 4.2.2 (no short shortcut):
=========================================================

zeta = pi/2 - 1.4 = 0.1708
pi - lambda6 - d* = 0.6957
2 * delta(1.949) = 0.7400 < 0.7416 = d* + zeta
pi - 1.949 - d* = 0.6218
 5 * 2 * 0.622 = 6.2200 < 2pi
\end{verbatim}